\documentclass[3p,11pt]{elsarticle}
\usepackage{amsmath}
\usepackage{graphicx}
\usepackage{caption} 
\usepackage{wrapfig}
\usepackage{color,soul}
\usepackage{enumitem}
\usepackage{subcaption}
\usepackage{comment}
\usepackage{amsmath}
\usepackage{times}
\usepackage{amssymb}
\usepackage{array}
\usepackage{yhmath}
\usepackage{latexsym}
\usepackage{amsthm}
\usepackage{amsfonts}
\usepackage{nomencl}
\usepackage[table]{xcolor}

\makenomenclature
\usepackage{hyperref}
\usepackage{array,tabularx}
\usepackage{scalerel,stackengine}
\usepackage{graphics}
\usepackage[english]{babel}
\allowdisplaybreaks

\stackMath
\numberwithin{equation}{section}
\date{}
\newtheorem{theorem}{Theorem}[section]
\newtheorem{problem}{Problem}[section]

\newtheorem{lemma}{Lemma}[section]
\newtheorem{corollary}{Corollary}[section]

\usepackage[utf8]{inputenc}

\begin{document}
%\linenumbers
\begin{frontmatter} 
\title{Supersonic–sonic patch solution for the two-dimensional Euler equations with a van der Waals equation of state \tnoteref{mytitlenote}} 
\author[1]{Anamika Pandey}
\author[1]{T. Raja Sekhar}
\address[1]{Department of Mathematics, Indian Institute of Technology Kharagpur, Kharagpur, West Bengal, India} 
\cortext[mycorrespondingauthor]{Corresponding author} 
\ead{trajasekhar@maths.iitkgp.ac.in} 

\begin{abstract}
We investigate supersonic transonic phenomena in the two-dimensional compressible Euler equations governed by a polytropic van der Waals equation of state. In contrast to the ideal gas setting, the non-ideal pressure law introduces stronger nonlinear effects and modifies the degeneracy structure near sonic states, which significantly complicates the analytical treatment of transonic flows. Within the self-similar framework associated with the four-state Riemann problem, we construct a supersonic sonic patch solution that connects a strictly supersonic region to a sonic boundary along a pseudo streamline. The analysis is based on a characteristic decomposition combined with a partial hodograph transformation, through which the problem is reformulated as a degenerate hyperbolic system. We establish the existence of a globally defined supersonic solution and prove its uniform regularity up to the sonic curve. In addition, we investigate the regularity properties of the resulting sonic boundary. Our results extend the theory of supersonic sonic patches from polytropic gases to a realistic non-ideal gas model.
\end{abstract}

% REQUIRED
\end{frontmatter}
\section{Introduction}

%%%%%%%%%%%%%%%%%% Physical motivation from non-ideal gases
The mathematical theory of compressible fluid dynamics has traditionally relied on idealized constitutive laws, most notably the polytropic gas assumption, to investigate fundamental flow phenomena. While such models have led to profound analytical insights, they fail to capture several essential physical effects present in realistic gases, including intermolecular attraction and finite molecular volume. These effects become particularly pronounced in regimes involving high compression and strong nonlinear interactions such as those occurring near sonic transitions. A more realistic description is provided by the van der Waals equation of state, which modifies the pressure-density relation and consequently alters the characteristic structure of the governing equations. From both physical and mathematical perspectives, it is therefore natural to ask whether the existing analytical theory of transonic flows can be extended beyond the ideal gas framework.

In non-ideal gas models, the presence of additional nonlinear terms in the pressure law leads to significant changes in the behavior of the sound speed and characteristic directions. Near sonic states, these modifications strengthen the degeneracy of the governing equations and complicate the analysis of regularity and stability. Many techniques that are effective for polytropic gases rely delicately on the specific algebraic structure of the pressure function and cannot be transferred directly to van der Waals gases. As a result, the investigation of transonic phenomena in non-ideal gases presents new analytical challenges that have not yet been fully explored.

Transonic flows and the associated mixed-type partial differential equations occupy a central position in the study of compressible Euler systems. A key feature of such flows is the presence of sonic curves, across which the governing equations change type from strictly hyperbolic to elliptic or degenerate hyperbolic. The loss of strict hyperbolicity on these curves prevents the direct application of standard methods for hyperbolic systems and makes the analysis of solutions near sonic boundaries particularly delicate. Understanding the behavior of flows in the neighborhood of sonic curves is not only of theoretical interest but is also crucial in practical applications such as nozzle design, high-speed aerodynamics and propulsion systems, where smooth transitions are essential to avoid shock-induced instabilities and performance degradation.

A natural and effective framework for studying transonic phenomena is provided by self-similar solutions arising from multidimensional Riemann problems. These problems describe the evolution of piecewise constant initial data and serve as fundamental building blocks for more complex flow patterns. In two space dimensions, the four-state Riemann problem for the compressible Euler equations has attracted extensive attention as a prototype for understanding nonlinear wave interactions. The seminal work of Zhang and Zheng \cite{conjecture1990} proposed several possible global configurations of solutions, which were subsequently confirmed and refined through numerical investigations \cite{pdlax, Schulz1993, Glimm2008, Alexander2017}. These studies revealed that the resulting flow patterns may involve a rich combination of shocks, rarefaction waves and contact discontinuities. Remarkably, the solution structure of the 2-D Riemann problem incorporates several fundamental flow phenomena, including shock reflections and dam-break type motions \cite{Jli2009}. Further analytical and numerical investigations have shown that, except for the vacuum case near the origin \cite{Jli2010}, the global solutions of the two-dimensional Riemann problem inevitably contain transonic features and fine-scale structures \cite{Jli1998, yuxizhang2001}, which render rigorous mathematical treatment extremely challenging. In particular, numerical simulations reported in \cite{Glimm2008} indicate that shock formation may occur in the vicinity of sonic curves even when the initial data are purely rarefactive. This observation highlights the subtle and intricate nature of supersonic flows near sonic boundaries and underscores the need for a rigorous analytical understanding of such regions.

Related transonic phenomena also arise in the study of gas expansion into a vacuum and pseudo-steady supersonic flow past solid boundaries. The expansion of a semi-infinite wedge of gas into a vacuum, commonly referred to as the dam break problem in hydraulics, has been investigated using characteristic decompositions and hodograph transformations \cite{simplewave2006, gasexpansion2001, interaction2009, chdecomposition, LaiGeng2019}. Similarly, pseudo-steady supersonic flows past sharp corners have been studied in \cite{wsheng2018, glai2021, Rahul2}. In many of these problems, however, the presence of vacuum state or geometric singularities limit the applicability of the analysis near sonic curves, leaving several fundamental questions unresolved. Beyond these examples, many related efforts have been devoted to the 2-D compressible Euler system as well as numerous other related models for a variety of initial and boundary value problems; see, for instance, \cite{Rahul3, cshen, Pandey2, Pandey1}.

To capture the local behavior of supersonic solutions near sonic boundaries, Song and Zheng \cite{songzheng2009} introduced the notion of semi-hyperbolic waves in the context of the pressure gradient system. A semi-hyperbolic patch is a local solution in which one family of characteristic curves originates from a sonic curve and terminates at either another sonic curve or a transonic shock. Such patches play the role of transition layers between strictly hyperbolic regions and degenerate sonic boundaries. Semi-hyperbolic patch solutions arise naturally in a variety of physical settings, including transonic flow past airfoils \cite{Courant1948} and Guderley-type shock reflections associated with the von Neumann triple point paradox \cite{Tesdall2006, Tesdall2008}. The existence and regularity of semi-hyperbolic patches have been established for the two-dimensional compressible Euler system and several related systems in a series of works \cite{LMingjie2011, yanbohu2014, lgeng2015, Yongqiang2023}. Uniform regularity up to the sonic boundary has also been studied in \cite{yanbohu2018, chenj2020, CJianjun2025, Rahul4}. Recently, Hu et al. \cite{hu2024supersonic} considered a formulation in which data are prescribed only along a pseudo streamline and constructed a supersonic sonic patch solution. These developments suggest that such patch structures may play an important role in the eventual construction of global solutions to mixed-type systems.

Despite this progress, most existing results on semi-hyperbolic and supersonic sonic patches are developed under the assumption of ideal or polytropic gases. In these settings, the boundary data are typically prescribed along characteristic curves, allowing the use of level curves of the pseudo Mach number $M_a$ as effective Cauchy supports up to the sonic boundary. For non-ideal gases, however, the modified pressure law leads to changes in the characteristic geometry and degeneracy structure, making it unclear whether these techniques remain applicable. In particular, the stronger nonlinear effects introduced by the van der Waals equation of state require a careful reexamination of the analytical framework.

In this article, we focus on the construction of a supersonic sonic patch solution along a pseudo streamline for the two-dimensional compressible Euler equations with a van der Waals equation of state. The patch extends from a strictly supersonic region to a sonic curve in the self-similar plane. Unlike the semi hyperbolic case, the Cauchy supports cannot be chosen freely up to the sonic boundary, which necessitates new ideas to establish existence and uniform regularity. Our results extend the existing theory for polytropic gases to a more realistic non ideal gas model.

The two-dimensional compressible Euler equations under consideration are given by \cite{Courant1948}
\begin{equation}\label{int_1}
\begin{split}
    &\frac{\partial \rho}{\partial t}+\frac{\partial(\rho u)}{\partial x}+\frac{\partial (\rho v)}{\partial y}=0,\\
    &\frac{\partial (\rho u)}{\partial t}+\frac{\partial(\rho u^2+p) }{\partial x}+\frac{\partial (\rho uv)}{\partial y}=0,\\
    &\frac{\partial (\rho v)}{\partial t}+\frac{\partial(\rho uv) }{\partial x}+\frac{\partial (\rho v^2+p)}{\partial y}=0, \; (t,x,y) \in \mathbb{R}^+ \times \mathbb{R}^2,
\end{split}
\end{equation}
where $\rho$ denotes the density, $u$ and $v$ represent the velocity components in the $x$ and $y$ directions, respectively and $p$ is the pressure.

Further, we consider a polytropic van der Waals gas with the equation of state \cite{callen1998thermodynamics, Mzafar2025}
\begin{equation}\label{int_2}
p(\tau)=\frac{K}{(\tau-b)^{\gamma+1}}-\frac{a}{\tau^2},
\end{equation}
where $\tau=1/\rho$ denotes the specific volume, $\gamma\in(0,1)$ is a constant and $K>0$ depends on the entropy of the system. The constants $a>0$ and $b>0$ represent the intermolecular attraction and covolume effects, respectively. The case $a=0$ corresponds to a dusty gas while $a=0$ and $b=0$ reduce the model to the classical polytropic ideal gas.

The main analytical strategy of this article is based on a characteristic decomposition combined with a partial hodograph transformation. Through this transformation, the original problem is reformulated as a degenerate hyperbolic system in which existence and uniform regularity can be established up to the degenerate line. By exploiting the global invertibility of the transformation, we then construct a corresponding solution to the original formulation and show that it remains uniformly regular up to the sonic curve in the self-similar plane.

The remainder of this article is organized as follows. In Section 2, we introduce the characteristic variables and formulate the problem precisely, together with the statement of the main result. The partial hodograph transformation and the analysis of the resulting degenerate hyperbolic system, including the proof of existence and uniform regularity up to the degenerate line, are carried out in Section 3. In Section 4, we recover the solution in the original variables and establish its uniform regularity up to the sonic curve. Concluding remarks and future work are presented in Section 5.

%%%%%%%%%%%%%%%%%%%%%%%%%%%%%%%%%%%%%%%%%%%%%%%%%%%%%%%%%%%%%%%%%%%%%%%%%%%%%%%%%%%%%%

\section{Preliminary analysis and problem formation}

%%%%%%%%%%%%%%%%%%%%%%%%%%%%%%%%%%%%%%%%%%%%%%%%%%%%%%%%%%

%\subsection{System in $(\xi,\eta)-$plane}
\subsection{Euler system in self-similar coordinates}

The compressible Euler system \eqref{int_1} admits a natural scaling invariance with respect to the transformation $(x,y,t)\mapsto(\alpha_1 x,\alpha_1 y,\alpha_1 t)$ for any $\alpha_1>0$. This observation motivates the introduction of self-similar variables $(\xi,\eta)=(x/t,y/t)$, under which solution depends only on the similarity coordinates. In these variables, system \eqref{int_1} is rewritten as
\begin{equation}\label{Pre1_1}
\begin{split}
   &-\xi\rho_\xi+(\rho u)_\xi-\eta\rho_\eta+(\rho v)_\eta=0,\\
   &-\xi(\rho u)_\xi+(\rho u^2+p(\rho))_\xi-\eta(\rho u)_\eta +(\rho u v)_\eta=0,\\
   &-\xi(\rho v)_\xi+(\rho uv)_\xi-\eta(\rho v)_\eta+(\rho v^2+p(\rho))_\eta=0.
\end{split}
\end{equation}
Assuming that the flow is irrotational, i.e., $v_\xi=u_\eta$, we can rewrite system (\ref{Pre1_1}) as
\begin{equation}\label{Pre1_2}
    \begin{cases}
        (c^2-U^2)u_\xi-UV(u_\eta+v_\xi)+(c^2-V^2)v_\eta=0,\\
        u_\eta-v_\xi=0,
    \end{cases}
\end{equation}
where $(U,V)=(u-\xi,v-\eta)$ denote the pseudo-flow velocity in self-similar plane and $c^2=-\tau^2p'(\tau)$ denotes the speed of sound such that
\begin{equation}\label{int_3}
    p'(\tau)=\frac{-K(\gamma+1)}{(\tau-b)^{\gamma+2}}+\frac{2a}{\tau^3}, \;\;p''(\tau)=\frac{K(\gamma+1)(\gamma+2)}{(\tau-b)^{\gamma+3}}-\frac{6a}{\tau^4}.
\end{equation}
From the expressions of $p'(\tau)$ and $p''(\tau)$, it follows that there exists sufficiently large $\tau_1>b$ such that for all $\tau>\tau_1$,
$p'(\tau)<0 \quad \text{and} \quad p''(\tau)>0$.

Introducing the pseudo-velocity potential $\phi$, we obtain the pseudo-Bernoulli's law
\begin{equation}\label{Pre1_3}
    \frac{U^2+V^2}{2}+\frac{K}{(\tau-b)^{\gamma}}\left( \frac{\gamma+1}{\gamma}+\frac{b}{\tau-b}\right)-\frac{2a}{\tau} = -\phi,
\end{equation}
where the additive constant has been normalized to zero for convenience. The system (\ref{Pre1_2}) can be written in the matrix form as follows
\begin{equation}\label{Pre1_4}
   \begin{bmatrix}
       u \\ v
   \end{bmatrix}_\xi+ \begin{bmatrix}
       \frac{-2UV}{c^2-U^2} & \frac{c^2-V^2}{c^2-U^2}\\ -1 & 0
   \end{bmatrix} \begin{bmatrix}
       u \\ v
   \end{bmatrix}_\eta =0. 
\end{equation}
The two eigen values of (\ref{Pre1_4}) are $\lambda_\pm = \frac{UV\pm c\sqrt{U^2+V^2-c^2}}{U^2-c^2}$ with corresponding left eigenvectors $l_{\pm}=(1,\lambda_\pm)$. This implies that system (\ref{Pre1_2}) is of mixed type: it is supersonic when $U^2+V^2>c^2$, subsonic when $U^2+V^2<c^2$ and sonic when $U^2+V^2=c^2$. Let $M_a=\frac{\sqrt{U^2+V^2}}{c}$ denotes the pseudo Mach number. The set $\{(\xi,\eta) \;|\;M_a(\xi,\eta)=1 \}$ is referred to as the sonic curve. %Furthermore, a curve is called a pseudo streamline if its tangent direction at each point is parallel to the pseudo velocity vector $(U, V)$.
A curve whose tangent direction is parallel to the pseudo velocity vector $(U,V)$ is referred to as a pseudo streamline. 

Multiplying system (\ref{Pre1_4}) by $l_{\pm}$ we obtain the characteristic equations as
\begin{equation}\label{Pre1_5}
    \partial^\pm u+\lambda_\mp\partial^\pm v=0; \hspace{10mm}\partial^\pm=\partial_\xi+\lambda_{\pm} \partial_\eta.
\end{equation}
Now we define the concept of characteristic angles as in \cite{chdecomposition}. It is easy to see that the eigenvalues of the system (\ref{Pre1_4}) satisfy the relations
\begin{equation*}\label{Pre1_6}
    \tan \alpha=\lambda_+, \hspace{10mm} \tan \beta=\lambda_-,
\end{equation*}
where $\alpha$ and $\beta$ are the inclination angles corresponding to the positive and negative characteristic directions. Further we denote pseudo-Mach angle by $\omega$, which is the angle between positive (negative) characteristic and pseudo-velocity vector $(U, V)$ and pseudo-flow angle by $\theta$, which is the angle between pseudo-velocity vector and $\xi$-axis such that $\theta=\frac{\alpha+\beta}{2}, \omega=\frac{\alpha-\beta}{2}$, with the corresponding relations
\begin{equation}\label{Pre1_7}
    \tan \theta =\frac{V}{U}, \hspace{5mm} \sin \omega = \frac{c}{\sqrt{U^2+V^2}}.
\end{equation}
For convenience, we denote $\sin\omega=\bar\omega$ throughout the article.

Using \eqref{Pre1_7} together with the pseudo Bernoulli relation \eqref{Pre1_3}, the velocity components and sound speed are formulated in terms of $(\theta,\omega)$ as \cite{MR3342411}
\begin{equation}\label{Pre1_8}
    u=\xi+\frac{c \cos \theta}{\bar\omega}, \; v=\eta+\frac{c \sin \theta}{\bar\omega} \; \text{and}\; c(\tau) = \left[2\bar\omega^2\left(\frac{2a}{\tau}-\phi-\frac{K(\tau\gamma+(\tau-b))}{\gamma(\tau-b)^{\gamma+1}} \right) \right]^{1/2}.
\end{equation}
Further, the sonic curve $\{(\xi,\eta): M_a(\xi,\eta)=1\}$ is represented as $\{(\xi,\eta): \bar\omega(\xi,\eta)=1\}$. 

Moreover, we give the following notations, which are very important in the discussion of the article:
\begin{equation*}\label{Pre1_9}
    \begin{cases}
       & \kappa(\tau)=-\frac{2p'(\tau)}{2p'(\tau)+\tau p''(\tau)}=\frac{2-\frac{2b}{\tau}-\frac{4a(\tau-b)^{\gamma+3}}{K(\gamma+1)\tau^4}}{\gamma+\frac{2b}{\tau}+\frac{2a(\tau-b)^{\gamma+3}}{K(\gamma+1)\tau^4}},\\
    %&m(\tau)=\frac{\kappa(\tau)-1}{\kappa(\tau)+1}=\frac{2-\gamma-\frac{4b}{\tau}-\frac{2a(\tau-b)^{\gamma+3}}{K(\gamma+1)\tau^4}}{\gamma+2+\frac{6a(\tau-b)^{\gamma+3}}{K(\gamma+1)\tau^4}},\\
    &\mu^2(\tau)=\frac{1}{1+\kappa(\tau)}. 
    \end{cases}
\end{equation*}
From the above relations, it follows that when $\tau$ is sufficiently large, say $\tau_1>b$, when $\tau>\tau_1$ we have $p'(\tau)<0, p''(\tau)>0$ and $\kappa(\tau)>0$. %In addition, the function $c(\tau)$ remains bounded; that is, there exist positive constants  $c_1$ and $ c_2$ such that $c_1<c(\tau)<c_2$ holds for all $ \tau>\tau_1$, provided $\phi$ is bounded. 
Moreover, for technical convenience, we assume without loss of generality that $\kappa'(\tau)>0$ \cite{Rahul4, Mzafar2023}.

%%%%%%%%%%%%%%%%%%%%%%%%%%%%%%%%%%%%%%%%%%%%%%%%%%%%%%%%%%
\subsection{Problem formulation and statement of the main theorem}

In this subsection, we formulate the local supersonic–sonic patch problem in the $(\xi, \eta)$-plane and state the main theorem of the article.
%%%%%%%%%%%%%%%%%%%%% Problem formation %%%%%%%%%%%%%%%%%%%%%%%%%%%%%%
\begin{problem}\label{P1}
Let $\widehat{LM}$ be a smooth segment of a curve in the self-similar $(\xi,\eta)$-plane, along which the flow is strictly supersonic. We assume that $\widehat{LM}$ is a pseudo streamline and that one endpoint $M$ of this curve corresponds to a sonic state. The problem is to determine a smooth sonic curve emanating from the point $M$ and to construct a regular solution of the pseudo-steady Euler equations in the region bounded by the pseudo streamline $\widehat{LM}$ and the resulting sonic curve, in a neighborhood of $M$. A schematic illustration of the configuration is shown in Figure \ref{fig1}.
\end{problem}
The formulation of Problem \ref{P1} is motivated by the structure of self-similar solutions to the 2-D four-state Riemann problem for the compressible Euler equations, originally proposed by Zhang and Zheng \cite{conjecture1990}. In such problems, supersonic flows are connected to transonic or subsonic regions through pseudo streamlines and local supersonic–sonic patches arise as fundamental building blocks of the global solution. The present work focuses on the rigorous construction and regularity analysis of such a patch in the setting of a van der Waals gas.

The main result of this work can be stated as follows:
\begin{theorem}\label{int_th}
Let $\widehat{LM}$ be a smooth pseudo streamline given by $\eta=\psi(\xi)$ for $\xi\in[\xi_1,\xi_2]$, where $\psi$ is strictly decreasing and concave. Assume that the pseudo Mach number $M_a$ is strictly decreasing along $\widehat{LM}$ and satisfies $M_a=1$ at the endpoint $M=(\xi_2,\psi(\xi_2))$. Suppose further that $\psi'$ and $M_a$ are $C^2$ functions and are compatible with the pseudo Bernoulli's relation. Then there exists a sufficiently small smooth sonic curve $\widehat{MN}$ issuing from $M$ such that the pseudo-steady Euler system \eqref{Pre1_1} admits a supersonic–sonic patch solution $(\tau,u,v)(\xi,\eta)$ in a neighborhood of $M$, bounded by $\widehat{LM}$ and $\widehat{MN}$. Moreover, the sonic curve $\widehat{MN}$ is $C^{1,\mu}$ continuous and the solution $(\tau,u,v)$ is uniformly $C^{1,\mu}$ up to the sonic curve for any $\mu\in(0,1/3)$.
\end{theorem}
\begin{figure}
    \centering
    \includegraphics[width=0.8\linewidth]{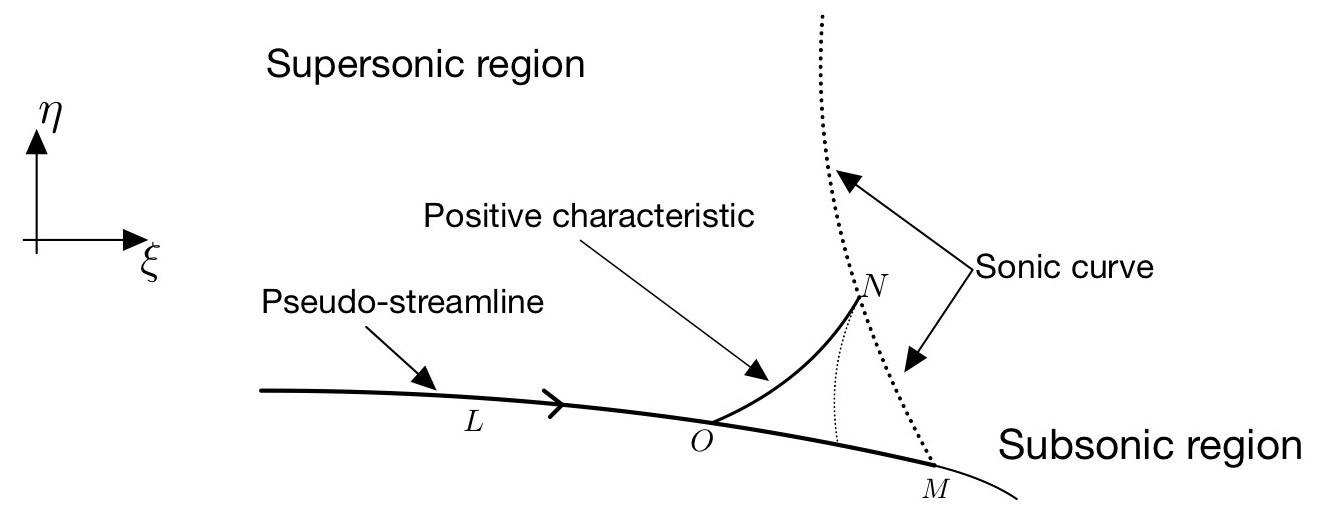}
    \caption{A supersonic-sonic patch in the self-similar plane}
    \label{fig1}
\end{figure}

%%%%%%%%%%%%%%%%%%%%%%%%%%%%%%%%%%%%%%%%%%%%%%%%%%%%%%%%%%
\subsection{Characteristic decompositions in terms of angle variables}
Define the first-order normalized directional derivatives $\bar{\partial}^\pm$ defined as
\begin{equation}\label{Pre2_1}
    \bar{\partial}^+ = \cos\alpha \partial_\xi+\sin\alpha \partial_\eta, \bar{\partial}^- = \cos\beta \partial_\xi+\sin\beta \partial_\eta, \bar{\partial}^0 = \cos\theta \partial_\xi+\sin\theta \partial_\eta.
\end{equation}
Further utilizing the first-order normalized directional derivatives $\bar{\partial}^\pm$ defined in (\ref{Pre2_1}) we obtain
\begin{equation}\label{Pre2_2}
    \begin{cases}
        &\partial_\xi = \cos\theta\bar{\partial}^0-\frac{\sin\theta}{2\bar\omega}(\bar{\partial}^+-\bar{\partial}^-),\\
        &\partial_\eta = \sin\theta\bar{\partial}^0+\frac{\cos\theta}{2\bar\omega}(\bar{\partial}^+-\bar{\partial}^-),\\
        &\bar\partial^0 = \frac{\bar\partial^++\bar\partial^-}{2\cos\omega}.
    \end{cases}
\end{equation} 
Using (\ref{Pre1_5}) and (\ref{Pre2_1}), we obtain a new system in terms of angle variables $(\theta,\omega)$ as
\begin{equation}\label{Pre2_3}
    \begin{split}
     \bar\partial^+\theta+\frac{\kappa(\tau)\cos\omega}{1+\kappa(\tau)\bar\omega^2} \bar\partial^+\bar\omega &=\frac{\bar\omega^2}{c}\cdot\frac{\kappa(\tau)-1-2\kappa(\tau)\bar\omega^2}{1+\kappa(\tau)\bar\omega^2},\\
        \bar\partial^-\theta-\frac{\kappa(\tau)\cos\omega}{1+\kappa(\tau)\bar\omega^2} \bar\partial^-\bar\omega &=\frac{\bar\omega^2}{c}\cdot\frac{1-\kappa(\tau)+2\kappa(\tau)\bar\omega^2}{1+\kappa(\tau)\bar\omega^2}.
    \end{split}
\end{equation}
%Furthermore utilizing the (\ref{Pre1_3}), (\ref{Pre1_8}) and (\ref{Pre2_1}) we obtain the expression for pseudo velocity potential as
Next, differentiating the pseudo-Bernoulli relation \eqref{Pre1_3} along the characteristic and pseudo-streamline directions yields
\begin{equation}\label{Pre2_4}
    \bar\partial^0\phi=-\frac{c}{\bar\omega}, \hspace{5mm}  \bar\partial^\pm\phi=-\frac{c\cos\omega}{\bar\omega}.
\end{equation}
To simplify later estimates, we introduce the normalized characteristic derivatives of the sound speed
\begin{equation}\label{Pre2_5}
    X=\frac{\bar\partial^+c}{c}, \hspace{5mm}  Y=\frac{\bar\partial^-c}{c}.
\end{equation}
Applying Bernoulli's law (\ref{Pre1_3}) we obtain the relation between $(X, Y)$ and $\omega$ as follows
\begin{equation}\label{Pre2_6}
    \begin{split}
        \bar\partial^+ \omega&= \tan\omega(1+\kappa(\tau)\bar\omega^2)X+\frac{\bar\omega^2}{c},\\
        \bar\partial^- \omega&= \tan\omega(1+\kappa(\tau)\bar\omega^2) Y+\frac{\bar\omega^2}{c}.
    \end{split}
\end{equation}
Also we obtain the relation between $c$ and $\bar\omega$ as
\begin{equation}\label{Pre2_7}
    \bar\partial^0c=\frac{c\bar\partial^0\bar\omega-\bar\omega^2}{\bar\omega(1+\kappa(\tau)\bar\omega^2)}.
\end{equation}
\begin{lemma}\label{Pre2_l1}
    The commutator relation between $\bar\partial^+$ and $\bar\partial^-$ can be expressed as \cite{Jliinteraction2009}
    \begin{equation*}
        \bar\partial^-\bar\partial^+-\bar\partial^+\bar\partial^-=\frac{1}{\sin(2\omega)}\left\{\left[\cos(2\omega)\bar\partial^-\alpha-\bar\partial^+\beta \right]\bar\partial^++ \left[\cos(2\omega)\bar\partial^+\beta-\bar\partial^-\alpha \right]\bar\partial^- \right\}.
    \end{equation*}
\end{lemma}

\begin{corollary} The characteristic decomposition of the variable $c$ is 
    \begin{equation}\label{Pre2_8}
    \begin{split}
        \bar\partial^-X &=X \left\{\frac{2\bar\omega\cos\omega}{c}+\frac{X+ Y}{2c\mu^2\cos^2\omega}-(2\bar\omega^2\kappa(\tau)-\tau k'(\tau)+1) Y \right\},\\
        \bar\partial^+ Y &= Y \left\{\frac{2\bar\omega\cos\omega}{c}+\frac{X+ Y}{2c\mu^2\cos^2\omega}-(2\bar\omega^2\kappa(\tau)-\tau k'(\tau)+1)X \right\}.
    \end{split}
\end{equation}
\end{corollary}
Denote
\begin{equation}\label{Pre2_9}
     \overline X=A(\kappa(\tau))\bar\omega\sqrt{1+\kappa(\tau)\bar\omega^2} X, \hspace{5mm}  \overline Y=-A(\kappa(\tau))\bar\omega\sqrt{1+\kappa(\tau)\bar\omega^2}  Y, 
\end{equation}
where $\displaystyle{A(\kappa(\tau))=exp\left(\int -\frac{\kappa'(\tau)\bar\omega^2}{2(1+\kappa(\tau)\bar\omega^2)} d\tau \right)}>0$. 

Using (\ref{Pre2_8}) and (\ref{Pre2_9}) we obtain 
\begin{equation}\label{Pre2_10}
    \begin{split}
        \bar\partial^- \overline X &= \overline X \left\{ \frac{1}{A(\kappa(\tau))\bar\omega \mu^2  \sqrt{1+\kappa(\tau)\bar\omega^2}} \cdot \frac{ \overline X- \overline Y}{2\cos^2\omega}+\frac{\bar\omega\cos\omega(3+4\kappa(\tau)\bar\omega^2)}{c(1+\kappa(\tau)\bar\omega^2)}\right \},\\
       \bar\partial^+ \overline Y &= \overline Y \left\{ \frac{1}{A(\kappa(\tau))\bar\omega  \mu^2\sqrt{1+\kappa(\tau)\bar\omega^2}} \cdot \frac{ \overline X- \overline Y}{2\cos^2\omega}+\frac{\bar\omega\cos\omega(3+4\kappa(\tau)\bar\omega^2)}{c(1+\kappa(\tau)\bar\omega^2)}\right \}.
    \end{split}
\end{equation}  
Further, for later use, we define some notations as follows
\begin{equation}\label{Pre2_11}
    \begin{split}
        \bar\partial^+\theta &= -\kappa(\tau)\cos\omega\bar\omega X-\frac{\bar\omega^2}{c},\\
        \bar\partial^-\theta &= \kappa(\tau)\cos\omega\bar\omega Y
        +\frac{\bar\omega^2}{c},\\
        \bar\omega_\xi &=\cos\theta\bar\omega(1+\kappa(\tau)\bar\omega^2){W}+\frac{\cos\theta\bar\omega^2}{c}-\sin\theta\frac{(1+\kappa(\tau)\bar\omega^2)}{2}( X- Y),\\
        \bar\omega_\eta &=\sin\theta\bar\omega(1+\kappa(\tau)\bar\omega^2){W}+\frac{\sin\theta\bar\omega^2}{c}+\cos\theta\frac{(1+\kappa(\tau)\bar\omega^2)}{2}( X- Y),
    \end{split}
\end{equation}
where ${W} =\frac{X+ Y}{2\cos\omega}$.
Also, we obtain 
\begin{equation}\label{Pre2_12}
    \phi_\xi\bar\omega_\eta-\phi_\eta\bar\omega_\xi =\frac{c({X}+{ Y})}{\sin(2\omega)}(1+\kappa(\tau)\bar\omega^2)=\frac{cW}{\bar\omega}(1+\kappa(\tau)\bar\omega^2).
\end{equation}

%%%%%%%%%%%%%%%%%%%%%%%%%%%%%%%%%%%%%%%%%%%%%%%%%%%%%%
\subsection{Sonic-supersonic boundary data and the main conclusion in terms of $(\theta, \bar\omega)$}

Consider a smooth curve $\widehat{MN}: \eta=\psi(\xi)(\xi \in[\xi_1,\xi_2])$ in $(\xi,\eta)$-plane. We prescribe the boundary condition on $\widehat{LM}$ for $(\tau,u,v)(\xi,\psi(\xi))=(\hat\tau,\hat u,\hat v)(\xi)$ such that
\begin{equation}\label{Pre3_1}
    \begin{split}
        & \hat\tau(\xi)>0, \hat v(\xi)-\psi(\xi)=\psi'(\xi)(\hat u(\xi)-\xi), \hspace{5mm} \forall \xi \in [\xi_1,\xi_2]\\
        & (\hat u(\xi)-\xi)^2+(\hat v(\xi)-\psi(\xi))^2> c^2(\hat\tau(\xi)) ,\hspace{5mm} \forall \xi \in [\xi_1,\xi_2]\\
        & (\hat u(\xi_2)-\xi_2)^2+(\hat v(\xi_2)-\psi(\xi_2))^2= c^2(\hat\tau(\xi_2)).
    \end{split}
\end{equation} 
These conditions ensure that $\widehat{LM}$ is a pseudo-streamline, along which the flow is strictly supersonic except at the endpoint $M=(\xi_2,\psi(\xi_2))$, where the flow reaches a sonic state. By exploiting (\ref{Pre1_7}) we obtain the data for $(c,\theta,\bar\omega)$ on $\widehat{LM}$
\begin{equation}\label{Pre3_2}
   \begin{split}
        c(\xi,\psi(\xi))&=\sqrt{\frac{K(\gamma+1)\hat\tau^2(\xi)}{(\hat\tau(\xi)-b)^{\gamma+2}}-\frac{2a}{\hat\tau(\xi)}}= \hat c(\xi),\\
        \theta(\xi,\psi(\xi))&=\arctan\left(\frac{\hat v(\xi)-\psi(\xi)}{\hat u(\xi)-\xi} \right)=\hat\theta(\xi), \\
        \bar\omega(\xi,\psi(\xi))&=\frac{\hat c(\xi)}{\sqrt{(\hat u(\xi)-\xi)^2+(\hat v(\xi)-\psi(\xi))^2}} =\hat{\bar\omega}(\xi), \hspace{5mm} \forall \; \xi\in[\xi_1,\xi_2].
   \end{split}
\end{equation}
Combining (\ref{Pre3_1}) and (\ref{Pre3_2}) we obtain
\begin{equation}\label{Pre3_3}
    \hat\theta(\xi)=\arctan( \psi'(\xi))\; \forall\; \xi\in[\xi_1,\xi_2], \hspace{5mm} \hat{\bar\omega}(\xi)<1 \; \forall \;\xi\in[\xi_1,\xi_2), \; \; \hat{\bar\omega}(\xi_2)=1.
\end{equation}
Further, we consider the compatibility condition for the functions $(\hat c, \hat\theta,\hat\omega)$ as follows
\begin{equation}\label{Pre3_4}  
    \cos\hat\theta \hat c'=\frac{\hat c\cos\hat\theta\hat{\bar\omega}'+\hat{\bar\omega}^2}{{\hat{\bar\omega}}(1+\kappa(\hat\tau){\hat{\bar\omega}^2})}, \;\;\; \forall \; \xi\in[\xi_1,\xi_2].
\end{equation}
Using the angle variables introduced above, we reformulate Problem \ref{P1} as an equivalent local boundary value problem.

\begin{problem}\label{P2}
Consider the smooth curve $\widehat{LM}: \eta=\psi(\xi)$ where $\xi_1\leq \xi\leq \xi_2$. Suppose the boundary conditions of $(c,\theta,\omega)|_{\widehat{LM}}=(\hat c,\hat\theta,\hat\omega)(\xi)$ defined in (\ref{Pre3_3}) and (\ref{Pre3_4}) are satisfied. The problem is to determine a smooth sonic curve $\widehat{MN}$ issuing from $M$ and to construct a supersonic solution $(c,\theta,\bar\omega)$ of system (\ref{Pre2_3}) in a neighborhood of $M$, bounded by the curves $\widehat{LM}$ and $\widehat{MN}$.
\end{problem}
To reflect the geometric configuration illustrated in Figure \ref{fig1}, we further assume that the functions $(\psi,\bar\omega)(\xi)$ satisfy
\begin{equation}\label{Pre3_5}
    \begin{split}
        &\psi'(\xi),{\hat{\bar\omega}}(\xi) \in C^2([\xi_1,\xi_2]),\\
        & \psi'(\xi_2)<0, \psi''(\xi_2)<0, \hat{\bar\omega}'(\xi_2)>0.
    \end{split}
\end{equation}
Since the construction is carried out in a sufficiently small neighborhood of the sonic point $M$, the conditions \eqref{Pre3_5} may equivalently be replaced by the uniform bounds
\begin{equation}\label{Pre3_6}
    \begin{split}
    &\psi'(\xi),{\hat{\bar\omega}}(\xi) \in C^2([\xi_1,\xi_2]),\\
        & \psi_0\leq -\psi'(\xi), -\psi''(\xi), \hat{\bar\omega}'(\xi)\leq \psi_1, \; \; \; \forall \;\xi\in[\xi_1,\xi_2],
    \end{split}
\end{equation}
for some positive constants $\psi_0$ and $\psi_1$.

We are now in a position to restate the main result in terms of the angle variables.

\begin{theorem}\label{Pre_th1}
Assume that the boundary data satisfies the conditions \eqref{Pre3_3}, \eqref{Pre3_4} and \eqref{Pre3_6}. Then there exists a sufficiently small smooth sonic curve $\widehat{MN}$ such that Problem \ref{P2} admits a supersonic solution $(c,\theta,\bar\omega)(\xi,\eta)\in C^2$ in a neighborhood of $M$ bounded by $\widehat{LM}$ and $\widehat{MN}$. Moreover, the sonic curve $\widehat{MN}$ is $C^{1,\mu}$ continuous and the solution $(c,\theta,\bar\omega)$ is uniformly $C^{1,\mu}$ up to $\widehat{MN}$ for any $\mu\in(0,1/3)$.
\end{theorem}
The sign conditions imposed in \eqref{Pre3_5} are chosen to match the local geometry of the pseudo-streamline near the sonic point. Other sign configurations may also be considered, provided they ensure the nondegeneracy conditions $\bar{\partial}^\pm c(M)\neq0$ and $\bar{\partial}^0 c(M)\neq0$, which are essential for the construction of supersonic--sonic patch solutions.

\subsection{The boundary data for $({X}, {Y})$}

We obtain the boundary data for $({X}, {Y})$ on $\widehat{LM}$ using the boundary data for $(\theta,\bar\omega)$. Exploring (\ref{Pre2_1}), (\ref{Pre3_3}) and (\ref{Pre3_6}) we obtain
\begin{equation}\label{Pre4_1}
    \begin{split}
        \bar\partial^0\theta|_{\widehat{LM}}&=\cos\hat \theta\theta'=\frac{\cos\hat \theta\psi''}{1+(\psi')^2}<0,\\
        \bar\partial^0\bar\omega|_{\widehat{LM}}&=\cos\hat \theta{\hat{\bar\omega}}'>0.
    \end{split}
\end{equation}
Utilizing (\ref{Pre2_2}) and (\ref{Pre2_3}) we obtain
\begin{equation}\label{Pre4_2}
    \begin{split}
        \bar\partial^+\omega &=\partial^0\bar\omega-\frac{1+\kappa(\tau)\bar\omega^2}{\kappa(\tau)\cos\omega}\bar\partial^0\theta,\\
        \bar\partial^-\omega &=\partial^0\bar\omega+\frac{1+\kappa(\tau)\bar\omega^2}{\kappa(\tau)\cos\omega}\bar\partial^0\theta.
    \end{split}
\end{equation}
Hence, usage of (\ref{Pre4_2}) together with (\ref{Pre2_6}) and (\ref{Pre2_8}) we obtain the expressions of $({X}, { Y})$ as follows
\begin{equation}\label{Pre4_3}
    \begin{split}
        {X} &= \frac{\cos\omega}{\bar\omega(1+\kappa(\tau)\bar\omega^2)} \left\{-\frac{\bar\omega^2}{c}+\partial^0\bar\omega-\frac{1+\kappa(\tau)\bar\omega^2}{\kappa(\tau)\cos\omega}\bar\partial^0\theta \right\}, \\
        { Y} &= \frac{\cos\omega}{\bar\omega(1+\kappa(\tau)\bar\omega^2)} \left\{-\frac{\bar\omega^2}{c}+\partial^0\bar\omega+\frac{1+\kappa(\tau)\bar\omega^2}{\kappa(\tau)\cos\omega}\bar\partial^0\theta \right\}. 
    \end{split}
\end{equation}
Therefore, combining (\ref{Pre4_1}) and (\ref{Pre4_3}) we obtain the boundary data of $({X}, { Y})$ on the pseudo-streamline $\widehat{LM}$ for $\xi\in[\xi_1,\xi_2]$
\begin{equation}\label{Pre4_4}
    \begin{split}
       {X}|_{\widehat{LM}}&=\frac{\cos\hat\omega}{\hat{\bar\omega}(1+\kappa(\hat\tau)\hat{\bar\omega}^2)} \left\{\frac{-\hat{\bar\omega}^2+\hat c\cos\hat\theta\hat{\bar\omega}'}{\hat c} -\frac{1+\kappa(\hat\tau)\hat{\bar\omega}^2}{\kappa(\hat\tau)\cos\hat\omega}\frac{\cos\hat\theta \psi''}{1+(\psi')^2}\right\}(\xi)=\hat a(\xi), \\
       { Y}|_{\widehat{LM}}&=\frac{\cos\hat\omega}{\hat{\bar\omega}(1+\kappa(\hat\tau)\hat{\bar\omega}^2)} \left\{ \frac{-\hat{\bar\omega}^2+\hat c\cos\hat\theta\hat{\bar\omega}'}{\hat c} + \frac{1+\kappa(\hat\tau)\hat{\bar\omega}^2}{\kappa(\hat\tau)\cos\hat\omega}\frac{\cos\hat\theta \psi''}{1+(\psi')^2}\right\}(\xi) =\hat b(\xi).
    \end{split}
\end{equation}
It follows directly from (\ref{Pre3_6}) together with the condition $\bar\omega(\xi_2)=1$ that there exists a number $\xi_0\in[\xi_1,\xi_2]$ such that $\hat a>0$ and $\hat b>0$ for all $\xi\in[\xi_0,\xi_2)$. We denote the corresponding point $(\xi_0,\psi(\xi_0))$ by $Q$.

Further, using (\ref{Pre2_5}), (\ref{Pre2_7}) and (\ref{Pre2_9}) we obtain the expression of $W$ as follows
\begin{equation}\label{Pre4_5}
    W =\frac{(c\bar\partial^0\bar\omega-\bar\omega^2)}{c\bar\omega(1+\kappa(\tau)\bar\omega^2)}.
\end{equation}
Thus, using (\ref{Pre4_1}) and (\ref{Pre4_5}) we obtain the boundary data for $W$ on $\widehat{LM}$
\begin{equation}\label{Pre4_6}
    W|_{\widehat{LM}}=\frac{\hat c\cos\hat\theta\hat{\bar\omega}'-\hat{\bar\omega}^2}{\hat c\hat{\bar\omega}(1+\kappa(\hat\tau)\hat{\bar\omega}^2)}(\xi)=\hat d(\xi).
\end{equation}
One readily verifies that $\hat d(\xi)>0$ for all $\xi\in[\xi_1,\xi_2]$.

Hence, by combining (\ref{Pre3_6}), (\ref{Pre4_4}) and (\ref{Pre4_6}), we derive the boundary conditions $({X},{ Y},W)$ on $\widehat{QM}$. Moreover, for certain constants $\hat m_0$ and $\hat M_0$ we have
\begin{equation}\label{Pre4_7}
    \begin{split}
        \hat a(\xi),\hat b(\xi) \in C^0([\xi_0,\xi_2])\cap C^1([\xi_0,\xi_2)), \hat d(\xi)\in C^1([\xi_0,\xi_2]),\\
        0<\hat m_0\leq \hat a(\xi),\hat b(\xi), \hat d(\xi)\leq \hat M_0, \hspace{5mm} \forall \xi\in[\xi_0,\xi_2],\\
        \hat a(\xi)+\hat b(\xi)=2\sqrt{1-\hat{\bar\omega}^2} \hat d(\xi), \hspace{5mm} \forall \xi\in[\xi_0,\xi_2].   
    \end{split}
\end{equation}
In addition, for $\xi\in[\xi_0,\xi_2)$ it follows from (\ref{Pre4_4}) and (\ref{Pre4_6})
\begin{equation}\label{Pre4_8}
    \begin{split}
        \sqrt{1-\hat{\bar\omega}^2}\hat a'(\xi) &= -\hat{\bar\omega}\hat{\bar\omega}'\hat d+(1-\hat{\bar\omega}^2)\hat d'-\sqrt{1-\hat{\bar\omega}^2}\left(\frac{\cos\hat\theta\psi''}{\kappa(\hat\tau)\bar\omega(1+(\psi')^2)} \right)',\\
        \sqrt{1-\hat{\bar\omega}^2}\hat b'(\xi) &= \hat{\bar\omega}\hat{\bar\omega}'\hat d+(1-\hat{\bar\omega}^2)\hat d'-\sqrt{1-\hat{\bar\omega}^2}\left(\frac{\cos\hat\theta\psi''}{\kappa(\hat\tau)\bar\omega(1+(\psi')^2)} \right)'.
    \end{split}
\end{equation}
Finally, the boundary values of the pseudo-velocity potential $\phi$ and its derivative on $\widehat{LM}$ follow from \eqref{Pre1_3} and \eqref{Pre2_4}:
\begin{equation*}\label{Pre4_9}
    \begin{split}
        \phi|_{\widehat{LM}}& = \frac{\hat c^2}{\hat{\bar\omega}^2}(\xi) -\frac{K}{(\hat\tau(\xi)-b)^\gamma}\left(\frac{\gamma+1}{\gamma}+\frac{b}{\hat\tau(\xi)-b} \right) +\frac{2a}{\hat\tau(\xi)}= \hat\phi(\xi),\\
        \phi'|_{\widehat{LM}}& =-\frac{\hat c}{\hat{\bar\omega}}\sqrt{1+\hat\psi'^2}(\xi)=\hat\phi' (\xi)<0  \; \; \; \forall \; \xi\in[\xi_1,\xi_2].
    \end{split}
\end{equation*}
In particular, $\hat\phi(\xi)$ is strictly decreasing along $\widehat{LM}$.

%%%%%%%%%%%%%%%%%%%%%%%%%%%%%%%%%%%%%%%%%%%%%%%%%%%%%%%%%%%%%%%%%%%%%%%%%%%%%%%%%%%%%%

\section{Solution in a partial hodograph plane}

The system (\ref{Pre2_10}) derived in Section 2 exhibits degeneracy along the sonic boundary, where the characteristic structure collapses and standard hyperbolic methods no longer can be applied directly. In the present setting, this degeneracy is further influenced by the nonlinear effects introduced through the van der Waals equation of state. To make the singular structure explicit and to facilitate the construction of solution near the sonic boundary, we introduce a partial hodograph transformation. This transformation converts the original system into a degenerate hyperbolic system in which the geometry of the degeneracy is fixed and the characteristic directions can be analyzed more effectively. We reformulate the problem (\ref{Pre2_10}) in the hodograph variables and establish the existence and regularity of solution under the boundary conditions (\ref{Pre4_7}) derived in Section 2.
 
\subsection{Reformulation in a partial hodograph plane}
We now introduce a partial hodograph transformation $(\xi,\eta) \mapsto (z,t)$ as 
\begin{equation}\label{Sol1_1}
    t=\cos\omega(\xi,\eta), \;\;\;z=\phi(\xi,\eta)-\hat\phi(\xi_2),
\end{equation}
where $\phi$ denotes the potential function appearing in pseudo-Bernoulli’s law (\ref{Pre1_3}).

From the boundary conditions established in (\ref{Pre3_6}) we recall that $\hat{\bar\omega}'(\xi)>0$ and $\hat\phi'(\xi)<0$ along the pseudo-streamline $\widehat{LM}$. Consequently, the image of the boundary segment $\widehat{MQ}$ under the transformation \eqref{Sol1_1} is a smooth increasing curve $\widehat{M'Q'}$ in the $(z,t)$-plane, parameterized as
\begin{equation}\label{Sol1_2}
    t=\sqrt{1-\hat{\bar\omega}^2(\xi)},\qquad 
    z=\hat\phi(\xi)-\hat\phi(\xi_2),
    \qquad \xi\in[\xi_0,\xi_2].
\end{equation}
 Let $\xi=\hat\xi(z)$ denote the inverse of $z=\hat\phi(\xi)-\hat\phi(\xi_2)$ for $z\in[0,z_0]$. Here $t_0=\sqrt{1-\hat{\bar\omega}^2(\xi_0)}\in(0,1)$ and $z_0=\hat\phi(\xi_0)-\hat\phi(\xi_2)>0$.

Hence, combining with (\ref{Pre4_4}) and (\ref{Pre4_6}), we achieve the boundary data of $({X}, {Y}, W)$ on $\widehat{M'Q'}$
\begin{equation}\label{Sol1_3}
\begin{split}
    {X}|_{\widehat{M'Q'}} &= \hat{a}(\hat{\xi}(z)) =: \hat{a}(z), \\
    { Y}|_{\widehat{M'Q'}} &= \hat{b}(\hat{\xi}(z)) =: \hat{b}(z), \\
   W|_{\widehat{M'Q'}} &=\hat{d}(\hat{\xi}(z)) =: \hat{d}(z) \quad \forall \;z \; \in [0,z_0].
\end{split}
\end{equation}
Furthermore, we also have by (\ref{Pre4_5})
\begin{equation}\label{Sol1_4}
\begin{split}
    &\hat{a}(z), \hat{b}(z), \hat{d}(z) \in C^1([0,z_0]), \\
    &\hat{m}_0 \leq \hat{a}(z), \hat{b}(z), \hat{d}(z) \leq \hat{M}_0, \quad \forall z \in [0,z_0],
\end{split}
\end{equation}
for positive constants $\hat m_0$ and $\hat M_0$.

\subsection{The degenerate system and uniform bounds}

In the hodograph variables $(z,t)$, the characteristic operators $\bar\partial^\pm$ take the form
\begin{equation}\label{Sol1_6}
\begin{split}
    \bar{\partial}^+ &= -\left\{ \frac{{(1+\kappa(\tau)-\kappa(\tau) t^2)(1-t^2)}\cdot {X}}{t} 
+ \frac{\sqrt{(1-t^2)^3}}{c}\right\} \partial_t + \frac{ct}{\sqrt{1-t^2}} \partial_z,\\
\bar{\partial}^- &= -\left\{ \frac{{(1+\kappa(\tau)-\kappa(\tau) t^2)(1-t^2)}\cdot { Y}}{t} 
+ \frac{\sqrt{(1-t^2)^3}}{c} 
\right\} \partial_t + \frac{ct}{\sqrt{1-t^2}} \partial_z.
\end{split}
\end{equation}
Using \eqref{Pre1_8} and (\ref{Sol1_6}) we obtain the expressions of $c, c_t$ and $c_z$ as follows
\begin{equation}\label{Sol1_7}
    \begin{split}
        \begin{cases}
             c(z,t) &= \left[2(1-t^2)\left(\frac{2a}{\tau}-z-\hat\phi(\xi_2)-\frac{K(\gamma\tau+(\tau-b))}{\gamma(\tau-b)^{\gamma+1}} \right) \right]^{1/2},\\
             c_t(z,t) &=-\frac{ct}{(1+\kappa(\tau)-\kappa(\tau) t^2)(1-t^2)},\\ 
             c_z(z,t) &=\frac{t^2-1}{c(1+\kappa(\tau)-\kappa(\tau) t^2)}.
        \end{cases}
    \end{split}
\end{equation}
Since $\tau(z,t)$ remains in a compact subset of $(b,\infty)$ in the local hodograph domain, the function $\mathcal{F}=\frac{2a}{\tau}
-\frac{K\big(\gamma\tau+(\tau-b)\big)}{\gamma(\tau-b)^{\gamma+1}}$ is bounded.
\begin{corollary}\label{cbound}
   The function $\mathcal{F}=\frac{2a}{\tau}
-\frac{K\big(\gamma\tau+(\tau-b)\big)}{\gamma(\tau-b)^{\gamma+1}}$ is bounded. Further, there exist positive constants $\hat{c}_0$ and $\hat{c}_1$ such that $\hat{c}_0\le c(z,t)\le \hat{c}_1$.
\end{corollary}
Thus we can deduce the system of variables $({X},{Y})(z,t)$ and $(\overline{X}, \overline{Y})(z,t)$ by (\ref{Pre2_8}) and (\ref{Pre2_10}) as follows
\begin{equation}\label{Sol1_8}
\begin{cases}
\partial_- {X} = -\frac{fX \sqrt{1-t^2}}{ Y-tg} \left\{\frac{2t^2\sqrt{1-t^2}}{c}+\frac{X+ Y}{2\mu^2t}+\left(\tau\kappa'(\tau)-1-2\kappa(\tau)+2t^2\kappa(\tau) \right)  Y t\right\},\\
\partial_+ {Y} =-\frac{f Y \sqrt{1-t^2}}{X-tg} \left\{\frac{2t^2\sqrt{1-t^2}}{c}+\frac{X+ Y}{2\mu^2t}+\left(\tau\kappa'(\tau)-1-2\kappa(\tau)+2t^2\kappa(\tau) \right) X t\right\},
\end{cases}
\end{equation}
and
\begin{equation}\label{Sol1_8a}
\begin{cases}
\partial_- \overline{X} = -\frac{A(\kappa(\tau))\overline X}{\sqrt{1-t^2}\mathcal{T}_1} \left\{\frac{1}{A(\kappa(\tau))\mu^2\sqrt{1+\kappa(\tau)(1-t^2)}}\cdot\frac{\overline X- \overline Y}{2t}+\frac{\sqrt{1-t^2}(3+4\kappa(1-t^2))}{c(1+\kappa(1-t^2))}t^2\right\},\\
\partial_+ \overline{Y} =-\frac{A(\kappa(\tau)) \overline Y}{\sqrt{1-t^2}\mathcal{T}_2} \left\{\frac{1}{A(\kappa(\tau))\mu^2\sqrt{1+\kappa(\tau)(1-t^2)}}\cdot\frac{\overline X- \overline Y}{2t}+\frac{\sqrt{1-t^2}(3+4\kappa(1-t^2))}{c(1+\kappa(1-t^2))}t^2\right\},
\end{cases}
\end{equation}
where
\begin{equation*}\label{Sol1_9}
\partial_\pm = \partial_t + \Lambda_{\pm}\partial_z, \;\;\;\Lambda_- = -\frac{cft^2}{ Y-tg}, \; \Lambda_+ = \frac{cft^2}{tg-X},
\end{equation*}
and
\begin{equation}\label{Sol1_10}
\begin{split}
    f(t,\tau) &=\frac{1}{(1+\kappa(\tau)-\kappa(\tau)t^2)(1-t^2)^{3/2}}\;, \; g(t,\tau) = -\frac{(1-t^2)^2f}{c},\\
    \mathcal{T}_1 &= \frac{A(\kappa(\tau)(Y-tg)}{f(1-t^2)},\;%,\;\\
    \mathcal{T}_2 = \frac{A(\kappa(\tau))(X-tg)}{f(1-t^2)}.
\end{split}
\end{equation}
Also, we can compute
\begin{equation*}\label{Sol1_11}
    \begin{split}
        f_t&=\left\{5\kappa(\tau)t(1-t^2)^{3/2}+3t\sqrt{1-t^2}-\kappa'(\tau)(1-t^2)^{5/2}\tau_t \right\}f^2, \; f_z=-(1-t^2)^{5/2}f^2\kappa'(\tau)\tau_z,\\
        g_t&=\frac{4t(1-t^2)f-(1-t^2)^2f_t}{c}+\frac{(1-t^2)^2f}{c^2}c_t,\; g_z=g\left(\frac{f_z}{f}-\frac{c_z}{c}\right),
    \end{split}
\end{equation*}
where $\tau_t=-\frac{\tau\kappa(\tau)}{c}c_t$ and $\tau_z=-\frac{\tau\kappa(\tau)}{c}c_z$.

Clearly, all the expressions of $f(t,\tau),g(t,\tau),c_t,c_z,\tau_t,\tau_z,f_t,f_z,g_t$ and $g_z$ are uniformly bounded in the region $L'M'N'$, near the sonic boundary $t=0$.

Obviously, system (\ref{Sol1_8}) is a closed system for $({X},{ Y})(z,t)$. For the degenerate boundary value problem (\ref{Sol1_7}), (\ref{Sol1_3}), we have the following theorem.
\begin{theorem}\label{Sol3_th1}
Suppose that (\ref{Sol1_3}) holds. Then there exists a positive real number $\bar\delta\in(0,t_0]$ such that the degenerate boundary value problem (\ref{Sol1_7}), (\ref{Sol1_3}) admits a smooth solution $({X},{Y})(z,t)$ in the whole region $M'N'O'$, where $O'$ is the point $(z(\bar{\delta}), \bar\delta)$ on $\widehat{M'Q'}$ and $N'$  is the intersection point of the positive characteristic passing through $O'$ and the line $t=0$. Moreover, the quantities $({X},{ Y})(z,t)$ are uniformly $C^{1-\nu}$ and $W(z,t)$ is uniformly $C^{2-\nu/3}$ up to the degenerate line $\widehat{M'N'}$ for $\nu\in(0,1).$
 \end{theorem}

\subsection{Construction of a determinate domain and the $C^0$-estimate}

To establish Theorem \ref{Sol3_th1}, we first construct a well-defined determinate domain for the nonlinear system (\ref{Sol1_7}) and then derive suitable a priori estimates for the solution within this domain.

We restrict attention to a neighborhood of the degenerate line. Let $\delta_0 = \min\{t_0,\, 1/\sqrt{2}\}$ then for any $(z,t) \in [0,z_0] \times [0,\delta_0]$ we have by (\ref{Sol1_7}) 
\begin{equation*}\label{Sol2_1}
\hat{c}_0 \leq c(z,t) \leq \hat{c}_1.
\end{equation*}
Assume that the transformed variables satisfy
\begin{equation}\label{Sol2_4}
\frac{1}{2}\hat{m}_0 \leq \overline{X},\overline{ Y} \leq 2\hat{M}_0.
\end{equation}
Under this assumption, the quantities $\mathcal{T}_1$ and $\mathcal{T}_2$  appearing in the characteristic system \eqref{Sol1_8a} can be estimated directly. For $t \in [0,\delta_1]$
\begin{equation*}\label{Sol2_5}
|\mathcal{T}_1|,|\mathcal{T}_2| \geq \sqrt{1+\kappa(\tau)(1-t^2)}\cdot \frac{1}{2}\hat{m}_0 - \frac{A(\kappa(\tau))(1-t^2)t}{\hat{c}_0} 
\geq \frac{\sqrt{\hat\kappa}}{2}\hat{m}_0 - \frac{\hat A\delta_1}{\hat{c}_0} \geq \frac{\sqrt{\hat\kappa}\hat{m}_0}{4},
\end{equation*}
where the constants are given as
\begin{equation*}\label{Sol2_2}
\hat{M} = 1 + \frac{4\hat A(3+4\hat\kappa)}{\hat{c}_0\sqrt{\hat\kappa}(1+\frac{\hat\kappa}{2})\hat{m}_0}, 
\qquad 
\delta_1 = \min\left\{\delta_0,\, \frac{\hat{c}_0\hat{m}_0\sqrt{\hat\kappa}}{4\hat A},\, \frac{\ln 2}{\hat{M}}\right\}.
\end{equation*}
Here, $\hat \kappa$ and $\hat A$ are positive constants. By this choice, we have  
\begin{equation*}\label{Sol2_3}
\delta_1 \leq \frac{\hat{c}_0\hat{m}_0\sqrt{\hat\kappa}}{4\hat A}, 
\qquad 
e^{\hat{M}\delta_1} \leq 2.
\end{equation*}
Using the above bounds together with the structure of the transformed system, we deduce that the nonlinear coefficients in the last two terms of (\ref{Sol1_7}) satisfy 
\begin{equation}\label{Sol2_6}
\left|\frac{A(\kappa(\tau))(3+4\kappa(\tau)(1-t^2))}{c(1+\kappa(\tau)(1-t^2))\mathcal{T}_2}\right|, 
\quad \left|-\frac{A(\kappa(\tau))(3+4\kappa(\tau)(1-t^2))}{c(1+\kappa(\tau)(1-t^2))\mathcal{T}_1}\right| 
\leq \frac{\hat\kappa\hat A(3+4\hat\kappa)}{\hat{c}_0(1+\tfrac{\hat\kappa}{2})\tfrac{\sqrt{\hat\kappa}\hat{m}_0}{4}} 
< \hat{M}.
\end{equation}
In addition, from the expressions of $\Lambda_{\pm}$ given in (\ref{Sol1_8}) and the estimate (\ref{Sol2_6}), we further obtain
\begin{equation*}\label{Sol2_7}
-\frac{\Lambda_{-}}{t^2}, \ \frac{\Lambda_{+}}{t^2} \leq 
\frac{\hat A \hat{c}_1}{\tfrac{1}{2}\frac{\sqrt{\hat \kappa}\hat{m}_0}{4}} 
= \frac{8\hat A\hat{c}_1}{\hat{m}_0\sqrt{\hat\kappa}} =: \widetilde{M}.
\end{equation*}
We now proceed to the geometric construction of the solution domain. 
Along the image of the boundary curve $\widehat{M'Q'}\cap\{t\le \delta_1\}$, 
a direct differentiation of \eqref{Sol1_2} yields
\begin{equation}\label{Sol2_8}
\begin{aligned}
\tilde{z}'(t) = \frac{\hat{c} \sqrt{1+(\psi')^2}}{\hat{\bar\omega}^2 \hat{\bar\omega}'}\, t.
\end{aligned}
\end{equation}
By the boundary assumptions \eqref{Pre3_6}, there exists a positive constant $\tilde m=\frac{\hat{c}_0 \sqrt{1+\psi_0'^2}}{\psi_1}>0$ such that
\begin{equation*}\label{Sol2_9}
\frac{\hat{c}\sqrt{1+(\psi')^2}}{\hat{\bar\omega}^2 \hat{\bar\omega}'}(\xi) \geq \tilde{m}, 
\qquad \forall \; \xi \in [\xi_1,\xi_2].
\end{equation*}
Define
\begin{equation}\label{Sol2_10}
     \delta = \min\left\{t_0,\, \frac{1}{\sqrt{2}},\, \frac{\hat{c}_0 \hat{m}_0\sqrt{\hat\kappa}}{4\hat A},\, 
      \frac{\ln 2}{\hat{M}},\, \frac{\tilde{m}}{\hat{M}}\right\}.
\end{equation}
Introducing the auxiliary curve
\begin{equation*}\label{Sol2_11}
    \bar z(t)=\tilde z(\delta)-\frac{\widetilde M}{3}\delta^3+\frac{\widetilde M}{3}t^3, \qquad t\in[0,\delta].
\end{equation*}
A direct comparison between \eqref{Sol2_8} and \eqref{Sol2_10} shows that $\bar z(0)\ge0$, which guarantees that the intersection point of the curve $z = \bar{z}(t)$ with the line $t=0$ lies to the right of the point $M'$. Denoting this intersection point by $N'=(\bar z(0),0)$ and setting $P'=(\tilde z(\delta),\delta)$, we define $\Omega=M'N'P'$. By construction, $\Omega$ is a strong determinate domain for \eqref{Sol1_7} whenever \eqref{Sol2_4} holds.

We aim to construct a solution to the degenerate boundary value problem \eqref{Sol1_7}–\eqref{Sol1_3} throughout the domain $\Omega$. It follows from the definition of $\delta$ in \eqref{Sol2_10} that the domain $\Omega$ is determined solely by the prescribed initial constants and is therefore independent of the particular procedure used to obtain the solution.

Moreover, the system \eqref{Sol1_7} is strictly hyperbolic away from the degenerate boundary and its coefficients remain smooth in a neighborhood of the point $P'$ in $\Omega$. Consequently, standard local existence results for quasilinear hyperbolic systems \cite{liboundaryvalue} ensure the existence of a $C^1$ solution in a neighborhood of $P'$. By the boundary conditions \eqref{Sol1_4}, this local solution satisfies the a priori bounds in \eqref{Sol2_4}.

To extend these bounds to the entire domain $\Omega$, we next derive uniform estimates. For any fixed $\epsilon\in(0,\delta]$, we denote by
\begin{equation*}
    \Omega_\epsilon:=\Omega \cap \{(z,t)|t\geq \epsilon\},
\end{equation*}
the truncated subdomain away from the degenerate line. We have the following lemma:
\begin{lemma}\label{Sol_3_l1}
     Assume that (\ref{Sol1_4}) holds and $(\overline{X},\overline{ Y})(z,t)$ be a $C^1$ solution of (\ref{Sol1_7}) and (\ref{Sol1_3}) in the domain $\Omega_\epsilon$. Then the solution $(\overline{X},\overline{ Y})(z,t)$ satisfies
     \begin{equation}\label{l11}
       \frac{1}{2}\hat{m}_0 \leq \overline{X}, \overline{Y} \leq 2\hat{M}_0.
     \end{equation}
\end{lemma}
\begin{proof}
We begin by establishing a uniform positive lower bound for the solution. 
To this end, we introduce a weighted transformation $(\overline{X}, \overline{Y})$ as
    \begin{equation*}\label{l12}
        \mathcal{X}=\overline{X}e^{-\hat M t}, \;\;  \mathcal{Y}=\overline{Y}e^{-\hat M t}.
    \end{equation*}
    We obtain the system \eqref{Sol1_8a} in terms of the new variables as follows
    \begin{equation}\label{l13} 
        \begin{split}
            \partial_-\mathcal{X} = -\hat M \mathcal{X}-\frac{A(\kappa(\tau))\mathcal{X}e^{\hat M t}}{\sqrt{1-t^2}\mathcal{T}_1} \left\{\frac{1}{A(\kappa(\tau))\mu^2\sqrt{1+\kappa(\tau)(1-t^2)}}\cdot\frac{\mathcal{X}- \mathcal{Y}}{2t}+\frac{\sqrt{1-t^2}(3+4\kappa(1-t^2))}{c(1+\kappa(1-t^2))}t^2e^{\hat M t}\right\},\\
             \partial_+\mathcal{Y} =-\hat M \mathcal{Y}-\frac{A(\kappa(\tau)) \mathcal{Y}e^{\hat M t}}{\sqrt{1-t^2}\mathcal{T}_2} \left\{\frac{1}{A(\kappa(\tau))\mu^2\sqrt{1+\kappa(\tau)(1-t^2)}}\cdot\frac{\mathcal{X}- \mathcal{Y}}{2t}+\frac{\sqrt{1-t^2}(3+4\kappa(1-t^2))}{c(1+\kappa(1-t^2))}t^2e^{\hat M t}\right\}.
        \end{split}
    \end{equation}  
    Fix $\epsilon'\in[\epsilon,\delta]$ and consider the level set $t=\epsilon'$, which we slide continuously downward from $t=\delta$ toward $t=\epsilon$. Assume, by contradiction, that there exists a first point $A\in\Omega_\epsilon$ where either $\mathcal{X}$ or $\mathcal{Y}$ attains the value $\hat m_0/2$. This point necessarily lies in the compact region bounded by the incoming characteristic curves issuing from $P'$, the initial boundary $\widehat{P'M'}$ and the horizontal segment $t=t_A$.

    By symmetry of the system, it suffices to consider the case
    \[
      \mathcal{X}(A)=\frac{\hat{m}_0}{2}.
    \]
    From the point $A$, we trace the negative and positive characteristic curves backward until they intersect the boundary $\widehat{P'M'}$, denoting the intersection points by $A_{-}$ and $A_{+}$, respectively. Then it follows that
    \[
      \mathcal{X}>\frac{\hat{m}_0}{2}, 
      \qquad 
      \mathcal{Y}>\frac{\hat{m}_0}{2}
      \quad \text{on} \; \; \widehat{AA_+}\setminus\{A\}.
    \]
     Consequently, the characteristic derivative satisfies $\partial_-\mathcal{X}(A)\ge 0$.
     
     However, evaluating first equation of \eqref{l13} at $A$, we find
     \begin{equation*}\label{l14}
         \begin{split}
             \partial_-\mathcal{X} = &-\hat M \mathcal{X}|_A-\frac{A(\kappa(\tau))\mathcal{X}e^{\hat M t}}{\sqrt{1-t^2}\mathcal{T}_1}\Bigg|_A \\ & \cdot\left(\frac{1}{A(\kappa(\tau))\mu^2\sqrt{1+\kappa(\tau)(1-t^2)}}\Bigg|_A\cdot\frac{\frac{\hat m_0}{2}- \mathcal{Y}|_A}{2t_A}+\frac{\sqrt{1-t^2}(3+4\kappa(1-t^2))t^2e^{\hat M t}}{c(1+\kappa(1-t^2))}\Bigg|_A\right)<0,
         \end{split}
     \end{equation*} 
      which contradicts the assumption $\partial_+\mathcal{X}|_A\geq 0$. Therefore, the lower bound holds:
    \begin{equation}\label{l15}
        \overline{X}, \overline{Y} \geq \frac{\hat m_0}{2}e^{\hat M t}\geq \frac{\hat m_0}{2}, \;\; \forall \; (z,t)\in \Omega_\epsilon.
    \end{equation}
    We now establish the upper bound. Define the alternative weighted variables
    \begin{equation*}\label{l16}
        \mathcal{X}=\overline{X}e^{\hat M t}, \;\;  \mathcal{Y}=\overline{Y}e^{\hat M t}.
    \end{equation*}
    We obtain system \eqref{Sol1_8a} in terms of $(\mathcal{X},\mathcal{Y})$  as follows
    \begin{equation}\label{l17} 
        \begin{split}
            \partial_-\mathcal{X} = \hat M \mathcal{X}-\frac{A(\kappa(\tau))\mathcal{X}e^{-\hat M t}}{\sqrt{1-t^2}\mathcal{T}_1} \left\{\frac{1}{A(\kappa(\tau))\mu^2\sqrt{1+\kappa(\tau)(1-t^2)}}\cdot\frac{\mathcal{X}- \mathcal{Y}}{2t}+\frac{\sqrt{1-t^2}(3+4\kappa(1-t^2))}{c(1+\kappa(1-t^2))}t^2e^{-\hat M t}\right\},\\
             \partial_+\mathcal{Y} =\hat M \mathcal{Y}-\frac{A(\kappa(\tau)) \mathcal{Y}e^{-\hat M t}}{\sqrt{1-t^2}\mathcal{T}_2} \left\{\frac{1}{A(\kappa(\tau))\mu^2\sqrt{1+\kappa(\tau)(1-t^2)}}\cdot\frac{\mathcal{X}- \mathcal{Y}}{2t}+\frac{\sqrt{1-t^2}(3+4\kappa(1-t^2))}{c(1+\kappa(1-t^2))}t^2e^{-\hat M t}\right\}.
        \end{split}
    \end{equation}  
    On the boundary $\widehat{P'M'}$, the initial bounds imply
    \begin{equation*}\label{l18}
        \mathcal{X}|_{\widehat{P'M'}}, \; \mathcal{Y}|_{\widehat{P'M'}} \leq \hat M_0e^{\hat M \delta}< 2\hat M_0.
    \end{equation*}

    Assume that there exists a first point $A$ in $\Omega_\epsilon$ such that $\mathcal{X}(A)=2\hat{M}_0$ and $\mathcal{X}, \mathcal{Y}<2\hat M_0$ on $\widehat{AA_+}\setminus\{A\}$, then we have $\partial_-\mathcal{X}|_A\leq 0$.
    Evaluating \eqref{l17} at $A$ we obtain
    \begin{equation*}\label{l19}
         \begin{split}
             \partial_-\mathcal{X} = &\hat M \mathcal{X}|_A+\frac{A(\kappa(\tau))\mathcal{X}e^{-\hat M t}}{\sqrt{1-t^2}\mathcal{T}_1}\Bigg|_A \\ & \cdot\left(\frac{1}{A(\kappa(\tau))\mu^2\sqrt{1+\kappa(\tau)(1-t^2)}}\Bigg|_A\cdot\frac{\frac{\hat m_0}{2}- \mathcal{Y}|_A}{2t_A}+\frac{\sqrt{1-t^2}(3+4\kappa(1-t^2))t^2e^{\hat M t}}{c(1+\kappa(1-t^2))}\Bigg|_A\right)>0,
         \end{split}
     \end{equation*}
     which is the contradiction as $\partial_-\mathcal{X}|_A\leq 0$. Hence,
     \begin{equation}\label{l110}
        \overline{X}, \overline{Y} \leq 2M_0e^{-\hat M t}\leq 2\hat M_0, \;\; \forall \; (z,t)\in \Omega_\epsilon.
    \end{equation}
    Combining \eqref{l15} and \eqref{l110}, we conclude that the desired uniform bounds hold throughout $\Omega_\epsilon$. This completes the proof.
\end{proof}
%%%%%%%%%%%%%%%%%%%%%%%%%%%%%%%%%%%%%%%%%%%%%%%%%%%%%%%%%%%%%%%%%%%%%%%%%%

\subsection{Existence of global solution in $\Omega$}

We now address the extension of the locally constructed $C^{1}$ solution near the point $P'$ to the entire determinate domain $\Omega$. To achieve this, it suffices to derive uniform a priori $C^{1}$ estimates for solutions of the degenerate boundary value problem \eqref{Sol1_7}–\eqref{Sol1_3} away from the sonic line.

In the view of Lemma \ref{Sol_3_l1}, if the solution of (\ref{Sol1_7}), (\ref{Sol1_3}) exists, it is appropriate to introduce
\begin{equation}\label{Sol3_1}
    \widetilde X = \frac{1}{ X}\; \text{and} \; \widetilde Y = -\frac{1}{ Y}.
\end{equation}
The system \eqref{Sol1_8} can be transformed in terms of new variables as
\begin{equation}\label{Sol3_2}
    \begin{split}
        \begin{cases}
            \tilde\partial_+\widetilde X&=\frac{f\sqrt{1-t^2}}{1+tg\widetilde Y}\left( \frac{\widetilde X-\widetilde Y}{2\mu^2 t}-\frac{2t^2\sqrt{1-t^2}\widetilde X\widetilde Y}{c}+\left(\tau\kappa'(\tau)-1-2\kappa(\tau)+2t^2\kappa(\tau)\right)\widetilde X t \right),\\
            \tilde\partial_-\widetilde Y&=\frac{f\sqrt{1-t^2}}{1-tg\widetilde X}\left( -\frac{\widetilde X-\widetilde Y}{2\mu^2 t}+\frac{2t^2\sqrt{1-t^2}\widetilde X\widetilde Y}{c}+\left(\tau\kappa'(\tau)-1-2\kappa(\tau)+2t^2\kappa(\tau)\right)\widetilde Y t \right), 
        \end{cases}
    \end{split}
\end{equation}
where
\begin{equation}\label{Sol3_3}
   \tilde\partial_\pm  =\partial_t+\tilde\Lambda_\pm \partial_z, \; \tilde\Lambda_-=-\frac{cft^2\widetilde X}{1-tg\widetilde X}, \; \tilde\Lambda_+=\frac{cft^2\widetilde Y}{1+tg\widetilde Y}.
\end{equation}
Further we compute
\begin{equation}\label{Sol3_4}
    \frac{\tilde\partial_-\tilde\Lambda_+-\tilde\partial_+\tilde\Lambda_-}{\Lambda_+-\Lambda_-}=\frac{2}{t}+h,
\end{equation}
where 
\begin{equation*}
   \begin{split}
        h(t,z)=&\frac{f\sqrt{1-t^2}}{(1-tg\widetilde X)(1+tg\widetilde Y)} \left\{ \frac{(\widetilde X-\widetilde Y)g}{2\mu^2}-\frac{2t^3\sqrt{1-t^2}g\widetilde X\widetilde Y}{c}+t(\tau\kappa'(\tau)-1-2\kappa(\tau)+2t^2\kappa(\tau))\right\}+ \\ &\frac{(g+tg_t)(\widetilde X-\widetilde Y+2tg\widetilde X\widetilde Y}{(1-tg\widetilde X)(1+tg\widetilde Y)}+\frac{f_t}{f}+tg+\frac{cft^3g_z\widetilde X\widetilde Y}{(1-tg\widetilde X)(1+tg\widetilde Y)}.
   \end{split}
\end{equation*}
We can rearrange the system (\ref{Sol3_2}) in the form
\begin{equation}\label{Sol3_5}
    \begin{split}
        \begin{cases}
          \tilde\partial_+\widetilde X &= \frac{\widetilde X-\widetilde Y}{2t}+H_{11}(\widetilde X-\widetilde Y)+H_{12}t,\\
          \tilde\partial_-\widetilde Y &= -\frac{\widetilde X-\widetilde Y}{2t}+H_{21}(\widetilde X-\widetilde Y)+H_{22}t,
        \end{cases}
    \end{split}
\end{equation}
where
\begin{equation}\label{Sol3_6}
    \begin{split}
        H_{11}&= -\frac{\sqrt{1-t^2}\widetilde Y}{2c(1+\kappa-\kappa t^2)(1+tg\widetilde Y)},\;\; H_{21}= -\frac{\sqrt{1-t^2}\widetilde X}{2c(1+\kappa-\kappa t^2)(1-tg\widetilde X)},\\
        H_{12} &=\frac{(\widetilde X-\widetilde Y)(1+2\kappa-\kappa t^2)}{2(1+tg\widetilde Y)(1-t^2)(1+\kappa-\kappa t^2)}-\frac{f\sqrt{1-t^2}}{1+tg\widetilde Y}\left(\frac{2t\sqrt{1-t^2}\widetilde X\widetilde Y}{c}-\left(\tau\kappa'(\tau)-1-2\kappa(\tau)+2t^2\kappa(\tau)\right)\widetilde X \right),\\
        H_{22} &=-\frac{(\widetilde X-\widetilde Y)(1+2\kappa-\kappa t^2)}{2(1-tg\widetilde X)(1-t^2)(1+\kappa-\kappa t^2)}+\frac{f\sqrt{1-t^2}}{1-tg\widetilde X}\left(\frac{2t\sqrt{1-t^2}\widetilde X\widetilde Y}{c}-\left(\tau\kappa'(\tau)-1-2\kappa(\tau)+2t^2\kappa(\tau)\right)\widetilde Y \right).
        %T_3&=\sqrt{1+\kappa-\kappa t^2}-\frac{t\sqrt{1-t^2}\widetilde Y}{c\sqrt{1+\kappa-\kappa t^2}},\;\;
        %T_4=\sqrt{1+\kappa-\kappa t^2}+\frac{t\sqrt{1-t^2}\widetilde X}{c\sqrt{1+\kappa-\kappa t^2}}
    \end{split}
\end{equation}
Now, differentiating (\ref{Sol3_5}) we obtain
\begin{equation}\label{Sol3_7}
    \begin{split}
        \tilde\partial_-\tilde\partial_+\widetilde X &=G_{11}\tilde\partial_-\widetilde X+G_{12},\\
        \tilde\partial_+\tilde\partial_-\widetilde Y &=G_{21}\tilde\partial_+\widetilde X+G_{22},
    \end{split}
\end{equation}
where
\begin{equation*}\label{Sol3_8}
    \begin{split}
          G_{11}=&-\frac{1}{2t}+H_{11}+tH_{12\widetilde X},\;\; G_{21}=\frac{1}{2t}-H_{21}+tH_{22\widetilde Y},\\
        G_{12}=&\left(-\frac{1}{2t}-H_{11}+(\widetilde X-\widetilde Y)H_{11\widetilde Y}+tH_{12\widetilde Y}\right)\cdot \left(-\frac{\widetilde X-\widetilde Y}{2t}+H_{21}(\widetilde X-\widetilde Y)+H_{22}t\right)-\frac{\widetilde X-\widetilde Y}{2t^2}(1-2t^2H_{11t})\\ &+\tilde \Lambda_-(tH_{12z}+(\widetilde X-\widetilde Y)H_{11z})+H_{12}+tH_{12t},\\
        G_{22}=&\left(-\frac{1}{2t}+H_{21}+(\widetilde X-\widetilde Y)H_{21\widetilde X}+tH_{22\widetilde X}\right)\cdot \left(\frac{\widetilde X-\widetilde Y}{2t}+H_{11}(\widetilde X-\widetilde Y)+H_{12}t\right)+\frac{\widetilde X-\widetilde Y}{2t^2}(1-2t^2H_{21t})\\ &+\tilde \Lambda_+(tH_{22z}+(\widetilde X-\widetilde Y)H_{21z})+H_{22}+tH_{22t}.
    \end{split}    
\end{equation*}
Applying commutator relation defined in Lemma \ref{Pre2_l1} for $\widetilde X$ and $\widetilde Y$ and utilizing (\ref{Sol3_4}) we obtain
\begin{equation}\label{Sol3_9}
    \begin{split}
        \tilde\partial_-\tilde\partial_+\widetilde X-\tilde\partial_+\tilde\partial_-\widetilde X&=\left(\frac{2}{t}+h\right)(\tilde\partial_+\widetilde X-\tilde\partial_-\widetilde X),\\
        \tilde\partial_-\tilde\partial_+\widetilde Y-\tilde\partial_+\tilde\partial_-\widetilde Y&=\left(\frac{2}{t}+h\right)(\tilde\partial_+\widetilde Y-\tilde\partial_-\widetilde Y).
    \end{split}
\end{equation}
Further, substituting (\ref{Sol3_7}) into (\ref{Sol3_9}) we have
\begin{equation}\label{Sol3_10}
    \begin{split}
        \begin{cases}
            \tilde\partial_+\tilde\partial_-\widetilde X&=\widetilde G_{11}\tilde\partial_-\widetilde X+\widetilde G_{12},\\
            \tilde\partial_-\tilde\partial_+\widetilde Y&=\widetilde G_{21}\tilde\partial_+\widetilde Y+\widetilde G_{22},
        \end{cases}
    \end{split}
\end{equation}
where
\begin{equation*}\label{Sol3_11}
    \begin{split}
        \widetilde G_{11}&=G_{11}+\frac{2}{t}+h,\;\; \widetilde G_{21}=G_{21}+\frac{2}{t}+h,\\
        \widetilde G_{12}&=G_{12}-\left(\frac{2}{t}+h\right)\cdot\left(\frac{\widetilde X-\widetilde Y}{2t}+H_{11}(\widetilde X-\widetilde Y)+H_{12}t\right),\\
        \widetilde G_{22}&=G_{22}-\left(\frac{2}{t}+h\right)\cdot \left(-\frac{\widetilde X-\widetilde Y}{2t}+H_{21}(\widetilde X-\widetilde Y)+H_{22}t\right).
    \end{split}
\end{equation*}
\begin{lemma}\label{Sol_3_l2}
    Assume that (\ref{Sol1_4}) holds and let $(X, Y)(z, t)$ be a $C^{1}$ solution of (\ref{Sol1_3}) and (\ref{Sol1_8}) in the domain $\Omega_{\epsilon}$. Then the solution $(X, Y)(z, t)$ satisfies
    \begin{equation}
        ||(X, Y)||_{C^1(\Omega_\epsilon)} \leq \frac{\widetilde K}{\epsilon^3},
    \end{equation}
    for a positive constant $\widetilde K$ independent of $\epsilon$.
\end{lemma}
\begin{proof}
    In view of Lemma \ref{Sol_3_l1} and the explicit forms of $f,g$ and $H_{ij}$ $(i,j=1,2)$, the terms $h, H_{ij\widetilde X}, H_{ij\widetilde Y}, H_{ijz}$ and $H_{ijt}$ are uniformly bounded in the domain $\Omega_\epsilon$. Hence, there exists a positive constant $\widetilde K$ such that
    \begin{equation}\label{Sol3_12}
        |\widetilde G_{11}|, |\widetilde G_{21}| \leq \frac{\widetilde K}{t}, \; \; |\widetilde G_{12}|, |\widetilde G_{22}| \leq \frac{\widetilde K}{t^2}\;\;\forall \; (z,t)\in \Omega_\epsilon.
    \end{equation}
    Fix an arbitrary point $(z,t)\in \Omega_\epsilon$. We extend the associated positive and negative characteristic curves backward until they intersect the boundary $\widehat{P'M'}$. Integrating system \eqref{Sol3_10} along these curves and invoking (\ref{Sol3_12}) yields
    \begin{equation}\label{Sol3_13}
        |\tilde\partial_-\widetilde X|, |\tilde\partial_+\widetilde Y|\leq \frac{\widetilde K}{\epsilon},
    \end{equation}
    where constant $\widetilde K$ is independent of $\epsilon$. By using (\ref{Sol3_13}) together with (\ref{Sol3_5}) and applying Lemma \ref{Sol_3_l1}, we obtain
    \begin{equation}\label{Sol3_14}
         |\tilde\partial_\pm\widetilde X|, |\tilde\partial_\pm\widetilde Y|\leq \frac{\widetilde K}{\epsilon}.
    \end{equation}
    Next, invoking the relations in \eqref{Sol3_3}, we express the physical derivatives in terms of characteristic derivatives as
    \begin{equation*}\label{Sol3_15}
        \partial_t = \frac{\widetilde X(1+tg\widetilde Y)\tilde\partial_++\widetilde Y(1-tg\widetilde X)\tilde\partial_-}{\widetilde X+\widetilde Y},\;\; \partial_z = \frac{(1+tg\widetilde Y)(1-tg\widetilde X)}{cf(\widetilde Y+\widetilde X)}\cdot\frac{\tilde\partial_+-\tilde\partial_-}{t^2}.
    \end{equation*}
    Combining this with Lemma \ref{Sol_3_l1} and (\ref{Sol3_14}) yields
    \begin{equation}
         ||(\widetilde X, \widetilde Y)||_{C^1(\Omega_\epsilon)} \leq \frac{\widetilde K}{\epsilon^3}.
    \end{equation}
    Consequently, from (\ref{Sol3_1}) we derive the $C^{1}$ estimates of the solution, which completes the proof of the lemma.
\end{proof}
In view of Lemmas \ref{Sol_3_l1} and \ref{Sol_3_l2}, a standard continuation argument may be employed to extend the locally constructed $C^1$ solution near the point $P'$ to the entire determinate domain $\Omega$. Indeed, for any fixed $\epsilon>0$, each level set of $t$ serves as a Cauchy boundary for the truncated region $\Omega_\epsilon$. The admissible continuation step depends only on the boundary data and on the uniform $C^0$ and $C^1$ bounds of $(X, Y)$, which remain uniformly bounded in $\Omega_\epsilon$. Since $\Omega_\epsilon$ is compact, the extension procedure requires only finitely many steps. As $\epsilon$ is arbitrary, this yields a $C^1$ solution in $\Omega\setminus\{t=0\}$.
\begin{theorem}\label{Sol3_th2}
    Suppose that (\ref{Sol1_4}) holds then the boundary value problem (\ref{Sol1_3}) and (\ref{Sol1_8}) admits a $C^1$ solution $(X, Y)(z, t)$ in the domain $\Omega\setminus\{t=0\}$.
\end{theorem}
%%%%%%%%%%%%%%%%%%%%%%%%%%%%%%%%%%%%%%%%%%%%%%%%%%%%%%%%%%%%%%%%%%%%%%%%%%%%%%%%%%%%%%%%%%%%

\subsection{Regularity of the solution}

We now investigate the regularity properties of the solution obtained in
Theorem \ref{Sol3_th2} as they approach the degenerate boundary $t=0$.
The main objective is to show that the solution admits a controlled extension
up to this boundary. To this end, it is sufficient to derive uniform bounds
for the spatial derivatives $(\widetilde X_z,\widetilde Y_z)$ in a neighborhood of
the degenerate line. We therefore define
\begin{equation*}\label{Sol4_1}
    \Xi=\tilde\partial_+\widetilde X-\tilde\partial_-\widetilde X,\; \Theta=\tilde\partial_+\widetilde Y-\tilde\partial_-\widetilde Y.
\end{equation*}
Then we have
\begin{equation}\label{Sol4_2}
    \tilde \Lambda_+-\tilde \Lambda_-=\frac{cf(\widetilde X+\widetilde Y)t^2}{(1-tg\widetilde X)(1+tg\widetilde Y)},\; \widetilde X_z=\frac{\Xi}{\tilde \Lambda_+-\tilde \Lambda_-}, \; \widetilde Y_z=\frac{\Theta}{\tilde \Lambda_+-\tilde \Lambda_-}.
\end{equation}
The commutator relation between $\tilde\partial_+$ and $\tilde\partial_-$ can be expressed as \cite{Jliinteraction2009}
\begin{equation*}\label{Sol4_3}
    \tilde\partial_-\tilde\partial_+-\tilde\partial_+\tilde\partial_-=\frac{\tilde\partial_-\tilde \Lambda_+-\tilde\partial_+\tilde \Lambda_-}{\tilde \Lambda_+-\tilde \Lambda_-}(\tilde\partial_+-\tilde\partial_-).
\end{equation*}
Using (\ref{Sol4_2}) we obtain
\begin{equation}\label{Sol4_4}
    \begin{split}
        \tilde\partial_+\Xi&=\frac{\tilde\partial_-\tilde\Lambda_+-\tilde\partial_+\tilde\Lambda_-}{\tilde\Lambda_+-\tilde\Lambda_-}\Xi+(\tilde\Lambda_+-\tilde\Lambda_-)(\tilde\partial_+\widetilde X)_z,\\
        \tilde\partial_-\Theta&=\frac{\tilde\partial_-\tilde\Lambda_+-\tilde\partial_+\tilde\Lambda_-}{\tilde\Lambda_+-\tilde\Lambda_-}\Theta+(\tilde\Lambda_+-\tilde\Lambda_-)(\tilde\partial_-\widetilde Y)_z.
    \end{split}
\end{equation}
Also, we can compute $(\tilde\partial_+\widetilde X)_z$ and $(\tilde\partial_-\widetilde Y)_z$ as follows
\begin{equation}\label{Sol4_5}
    \begin{split}
        (\tilde\partial_+\widetilde X)_z&=f_1\widetilde X_z+f_2\widetilde Y_z+f_3,\\
        (\tilde\partial_-\widetilde Y)_z&=g_1\widetilde Y_z+g_2\widetilde X_z+g_3,\\
    \end{split}
\end{equation}
where,
\begin{equation*}\label{Sol4_6}
    \begin{split}
        f_1(z,t)=&\frac{1}{2t}-\frac{(\kappa(\tau)+1)g\widetilde Y-t(1+2\kappa(\tau)-t^2\kappa(\tau))(1+tg\widetilde Y)}{2(1-t^2)(1+tg\widetilde Y)(1+\kappa(\tau)-t^2\kappa(\tau))}\\ &+\frac{f\sqrt{1-t^2}}{1+tg\widetilde Y} \left(t(\tau\kappa' (\tau)-1-2\kappa(\tau)+2t^2\kappa(\tau))-\frac{2t^2\sqrt{1-t^2}\widetilde Y}{c} \right),\\
         f_2(z,t)=&-\frac{1}{2t}+\frac{(\kappa(\tau)+1)g\widetilde Y-t(1+2\kappa(\tau)-t^2\kappa(\tau))(1+tg\widetilde Y)}{2(1-t^2)(1+tg\widetilde Y)(1+\kappa(\tau)-t^2\kappa(\tau))}-\frac{2t^2f(1-t^2)\widetilde X}{c(1+tg\widetilde Y)}\\ 
         &-\frac{fg\sqrt{1-t^2}}{(1+tg\widetilde Y)^2}\left(\frac{\widetilde X-\widetilde Y}{2\mu^2(\tau)}+(\tau\kappa'(\tau)-1-2\kappa(\tau)+2t^2\kappa(\tau))t^2\widetilde X-\frac{2t^3\sqrt{1-t^2}\widetilde X\widetilde Y}{c} \right),\\
         f_3(z,t)=& (\widetilde Y-\widetilde X)\left(\mu^2(\tau)\kappa'(\tau)\tau_z+\frac{f_z}{f(1+tg\widetilde Y)} \right)\left( \frac{(\kappa(\tau)+1)g\widetilde Y-t(1+2\kappa(\tau)-t^2\kappa(\tau))(1+tg\widetilde Y)}{2(1+\kappa(\tau)-t^2\kappa(\tau))(1-t^2)(1+tg\widetilde Y)}\right)\\ 
         &+ \frac{\widetilde X-\widetilde Y}{2t}\left(\mu^2(\tau)\kappa'(\tau)\tau_z+\frac{f_z}{f(1+tg\widetilde Y)} \right)+\frac{f\sqrt{1-t^2}}{1+tg\widetilde Y}\left(-\frac{2t^2\sqrt{1-t^2}\widetilde X\widetilde Y}{\tau\kappa(\tau)c}+\widetilde X t[\kappa'(\tau)(2t^2-1)+\tau\kappa''(\tau)]\right)\tau_z\\
         &+\frac{f_z\sqrt{1-t^2}}{(1+tg\widetilde Y)^2}\left(-\frac{2t^2\sqrt{1-t^2}\widetilde X\widetilde Y}{c}+\widetilde X t[\tau\kappa'(\tau)-1-2\kappa(\tau)+t^22\kappa(\tau))]\right) \\
         &+ \frac{fg\widetilde Y\sqrt{1-t^2}}{c(1+tg\widetilde Y)^2}\left(\frac{\widetilde X-\widetilde Y}{2\mu^2(\tau)}-\frac{2t^3\sqrt{1-t^2}\widetilde X\widetilde Y}{c}+\widetilde X t^2[\tau\kappa'(\tau)-1-2\kappa(\tau)+2t^2\kappa(\tau))]\right)c_z,\\
         g_1(z,t)=&\frac{1}{2t}+\frac{(\kappa(\tau)+1)g\widetilde X+t(1+2\kappa(\tau)-t^2\kappa(\tau))(1-tg\widetilde X)}{2(1-t^2)(1-tg\widetilde X)(1+\kappa(\tau)-t^2\kappa(\tau))}\\ &+\frac{f\sqrt{1-t^2}}{1-tg\widetilde X} \left(t(\tau\kappa' (\tau)-1-2\kappa(\tau)+2t^2\kappa(\tau))+\frac{2t^2\sqrt{1-t^2}\widetilde X}{c} \right),\\
         g_2(z,t)=&-\frac{1}{2t}-\frac{(\kappa(\tau)+1)g\widetilde X+t(1+4\kappa(\tau)-2t^2\kappa(\tau))(1-tg\widetilde X)}{2(1-t^2)(1-tg\widetilde X)(1+\kappa(\tau)-t^2\kappa(\tau))}+\frac{2t^2f(1-t^2)\widetilde Y}{c(1-tg\widetilde X)}\\ 
         &-\frac{fg\sqrt{1-t^2}}{(1-tg\widetilde X)^2}\left(\frac{\widetilde X-\widetilde Y}{2\mu^2(\tau)}-(\tau\kappa'(\tau)-1-2\kappa(\tau)+2t^2\kappa(\tau))t^2\widetilde Y-\frac{2t^3f\sqrt{1-t^2}\widetilde X\widetilde Y}{c} \right),\\
        g_3(z,t)=& (\widetilde Y-\widetilde X)\left(\kappa'(\tau)\mu^2(\tau)\tau_z+\frac{f_z}{f(1-tg\widetilde X)} \right)\left( \frac{(\kappa(\tau)+1)g\widetilde X+t(1+4\kappa(\tau)-2t^2\kappa(\tau))(1-tg\widetilde X)}{2(1+\kappa(\tau)-t^2\kappa(\tau))(1-t^2)(1-tg\widetilde X)}\right)\\ 
         &+ \frac{\widetilde Y-\widetilde X}{2t}\left(\kappa'(\tau)\mu^2(\tau)\tau_z+\frac{f_z}{f(1-tg\widetilde X)} \right)+\frac{f\sqrt{1-t^2}}{1-tg\widetilde X}\left(\frac{2t^2\sqrt{1-t^2}\widetilde X\widetilde Y}{\tau\kappa(\tau)c}+\widetilde Y t[\kappa'(\tau)(2t^2-1)+\tau\kappa''(\tau)]\right)\tau_z\\
         &+\frac{f_z\sqrt{1-t^2}}{(1-tg\widetilde X)^2}\left(\frac{2t^2\sqrt{1-t^2}\widetilde X\widetilde Y}{c}+\widetilde Y t[\tau\kappa'(\tau)-1-2\kappa(\tau)+t^22\kappa(\tau))]\right) \\
         &+ \frac{fg\widetilde X\sqrt{1-t^2}}{c(1-tg\widetilde X)^2}\left(\frac{\widetilde X-\widetilde Y}{2\mu^2(\tau)}-\frac{2t^3\sqrt{1-t^2}\widetilde X\widetilde Y}{c}-\widetilde Y t^2[\tau\kappa'(\tau)-1-2\kappa(\tau)+2t^2\kappa(\tau))]\right)c_z.
    \end{split}
\end{equation*}
By exploiting (\ref{Sol3_4}), (\ref{Sol4_2}), (\ref{Sol4_4}) and (\ref{Sol4_5}) we obtain
\begin{equation}\label{Sol4_8}
    \begin{split}
        \tilde\partial_+\Xi =\left(\frac{5}{2t}+\hat f_1(z,t)\right)\Xi+\hat f_2(z,t)\frac{\Theta}{2t}+\hat f_3(z,t)\frac{t}{2},\\
        \tilde\partial_-\Theta =\left(\frac{5}{2t}+\hat g_1(z,t)\right)\Theta+\hat g_2(z,t)\frac{\Xi}{2t}+\hat g_3(z,t)\frac{t}{2},
    \end{split}
\end{equation}
where
\begin{equation*}\label{Sol4_9}
    \begin{split}
        \hat f_1&=h+f_1-\frac{1}{2t}, \; \hat f_2=2tf_2, \; \hat f_3=\frac{2}{t}(\tilde\Lambda_+ -\tilde\Lambda_-)f_3,\\
        \hat g_1&=h+g_1-\frac{1}{2t} , \; \hat g_2=2tg_2, \; \hat g_3=\frac{2}{t}(\tilde\Lambda_+ -\tilde\Lambda_-)g_3.
    \end{split}
\end{equation*}
Noting the explicit forms of $h$, $f_i$ and $g_i$ $(i=1,2,3)$ and applying Lemma \ref{Sol_3_l1} we conclude that the functions $\hat f_1$, $\hat f_2$, $\hat f_3$, $\hat g_1$, $\hat g_2$ and $\hat g_3$ remain uniformly bounded in the domain $\Omega$. Moreover, as $t \to 0$, we observe that $\hat f_2 \to -1$ and $\hat g_2 \to -1$. That is, for a uniform constant $\hat K$ we have
\begin{equation*}\label{Sol4_10}
    |h|,|\hat f_i|,|\hat g_i| \leq \hat K, \; \; \forall \; (z,t)\in \Omega.
\end{equation*}
We next analyze the behavior of the spatial derivatives $\widetilde X_z$ and $\widetilde Y_z$ along the curve $\widehat{M'P'}$ in a neighborhood of the point $M'$. By differentiating $\widetilde X$ along $\widehat{M'P'}$ and invoking \eqref{Sol2_8} together with \eqref{Sol1_2}, we obtain
\begin{equation}\label{Sol4_11}
    \widetilde X_t|_{\widehat{M'P'}}+\frac{\hat c \sqrt{1+(\psi')^2}}{\hat{\bar\omega}^2\hat{\bar\omega}'}t\cdot\widetilde X_z|_{\widehat{M'P'}}=-\frac{\hat a'(\xi)\hat\xi'(t)}{\hat{a}^2}=\frac{\sqrt{1-\hat{\bar\omega}^2}\hat a'(\xi)}{\hat{\bar\omega}\hat{\bar\omega}' \hat{a}^2}.
\end{equation}
From (\ref{Sol4_11}) and exploring the first equation of (\ref{Pre4_8}) we obtain
\begin{equation*}\label{Sol4_12}
    \widetilde X_t|_{\widehat{M'P'}}+\frac{\hat c \sqrt{1+(\psi')^2}}{\hat{\bar\omega}^2\hat{\bar\omega}'}t\cdot\widetilde X_z|_{\widehat{M'P'}}=-\frac{\hat d}{\hat a^2}+\frac{t^2\hat d'}{\hat {\bar\omega}\hat{\bar\omega}'\hat a^2}+\frac{t}{\hat {\bar\omega}\hat{\bar\omega}'\hat a^2} \left(\frac{\cos\hat\theta\psi''}{\kappa(\hat\tau)\bar\omega(1+(\psi)'^2)} \right)'.
\end{equation*}
Now, utilizing the expression of $\widetilde X$ from (\ref{Sol3_5}) and the boundary data from (\ref{Sol1_3}) we obtain
\begin{equation*}\label{Sol4_13}
    \begin{split}
        \left(\frac{\hat c \hat ft}{t\hat g-\hat b}-\frac{\hat c \sqrt{1+(\psi')^2}}{\hat{\bar\omega}^2\hat{\bar\omega}'}\right)\cdot\tilde X_z|_{\widehat{M'P'}} =& \frac{\hat b+\hat a}{2t^2\hat a\hat b}+\frac{\hat b+\hat a}{\hat a\hat bt}\hat H_{11}+\hat H_{12}
        %& \frac{\hat f\sqrt{1-t^2}}{\hat a(\hat b+tg)t}\left(\frac{\hat b-\hat a}{2\hat \mu^2 t}-\frac{2t^2\sqrt{1-t^2}}{\hat c}+\hat b\left(\hat\tau\kappa'(\hat\tau)-1-2\kappa(\hat\tau)+2t^2\kappa(\hat\tau)\right)t \right)
        +\frac{\hat d}{t\hat a^2}-\frac{t\hat d'}{\hat {\bar\omega}\hat{\bar\omega}'\hat a^2}-\frac{1}{\hat {\bar\omega}\hat{\bar\omega}'\hat a^2} \left(\frac{\cos\hat\theta\psi''}{\kappa(\hat\tau)\bar\omega(1+(\psi)'^2)} \right)'\\
        =& \frac{2\hat d^2}{\hat a^2\hat b}+\frac{2\hat d}{\hat a\hat b}\hat H_{11}+\hat H_{12}-\frac{t\hat d'}{\hat{\bar\omega}\hat{\bar\omega}'\hat a^2}-\frac{1}{\hat {\bar\omega}\hat{\bar\omega}'\hat a^2} \left(\frac{\cos\hat\theta\psi''}{\kappa(\hat\tau)\bar\omega(1+(\psi)'^2)} \right)',
    \end{split}
\end{equation*}
which indicates that the boundary value $\widetilde X_{z}$ is uniformly bounded on the curve $\widehat{M'P'}$. This conclusion relies on the uniform positivity of $\hat f$ and $\hat g$, the uniform bounds for $\hat{H}_{11}$ and $\hat{H}_{12}$ and the relation $\widehat{a}+\hat{b}=2t\hat{d}$ holds. An analogous argument shows that $\widetilde Y_{z}$ is also uniformly bounded on $\widehat{M'P'}$.

The following lemma is a consequence of the uniform boundedness of $\widetilde X_{z}$ and $\widetilde Y_{z}$ along the curve $\widehat{M'P'}$.
\begin{lemma}\label{Sol_3_l3}
    Let $\nu\in(0,1)$ be fixed. Then the quantities $t^{\nu}|\widetilde  {X}_{z}|$ and $t^{\nu}| \widetilde Y_{z}|$ remain uniformly bounded throughout the entire domain $\Omega$ up to the degenerate line $t=0$.
\end{lemma}
\begin{proof} 
    Define the rescaled variables
\begin{equation}\label{Sol4_14}
\widetilde \Xi=\frac{\Xi}{t^{2-\nu}},\qquad \widetilde \Theta=\frac{\Theta}{t^{2-\nu}}.
\end{equation}
From Lemma \ref{Sol_3_l1} and (\ref{Sol4_2}), it suffices to show that $\widetilde \Xi$ and $\widetilde \Theta$ are uniformly bounded in $\Omega$.

The governing system for $(\widetilde \Xi, \widetilde \Theta)$ is
\begin{equation*}\label{Sol4_15}
\begin{split}
    \tilde{\partial}_+\widetilde \Xi
&=\left(\frac{1+2\nu}{2t}+\hat f_1\right)\widetilde\Xi
+\frac{\widetilde \Theta}{2t}\hat f_2+\frac{t^{\nu-1}}{2}\hat f_{3},\\
\tilde{\partial}_-\widetilde\Theta
&=\left(\frac{1+2\nu}{2t}+\hat g_1\right)\widetilde\Theta+\frac{\widetilde \Xi}{2t}\hat g_2+\frac{t^{\nu-1}}{2}\hat g_{3}.
\end{split}
\end{equation*}
We also evaluate
\begin{equation}\label{Sol4_16}
\begin{split}
    \tilde{\partial}_+\left(t^{-\frac{1}{2}-\nu}\widetilde \Xi \right)
&=t^{-\frac{3}{2}-\nu}\left(\hat f_1t\widetilde\Xi+\frac{\widetilde \Theta}{2}\hat f_2+\frac{t^{\nu}}{2}\hat f_{3}\right),\\
\tilde{\partial}_-\left(t^{-\frac{1}{2}-\nu}\widetilde\Theta \right)
&=t^{-\frac{3}{2}-\nu}\left(\hat g_1t\widetilde\Theta+\frac{\widetilde \Xi}{2}\hat g_2+\frac{t^{\nu}}{2}\hat g_{3}\right).
\end{split}
\end{equation}
Since $\widetilde X_{z}$ and $\widetilde Y_{z}$ are uniformly bounded on $\widehat{M'P'}$, (\ref{Sol4_8}) implies $(\widetilde \Xi, \widetilde\Theta)$ are uniformly bounded on $\widehat{M'P'}$.

Define
\begin{equation}\label{Sol4_17}
\overline C=\max_{\widehat{M'P'}}\{|\widetilde\Xi|, |\widetilde\Theta|\},\qquad
\bar{\epsilon}=\min\left\{\bar{\delta},\frac{\nu}{4\hat K}\right\}>0.
\end{equation}
The domain $\Omega$ is partitioned into two subregions
$\Omega_{1}:=\Omega\cap \{t\le \bar{\epsilon}\}$ and $\Omega_{2}:=\Omega\cap \{t\ge \bar{\epsilon}\}$. It is clear that the variables $(\widetilde \Xi, \widetilde\Theta)$ remain uniformly bounded throughout the region $\overline\Omega_{2}$, where $\overline\Omega_{2}$ denotes the closure of $\Omega_{2}$.

Let us introduce the notation
\begin{equation}\label{Sol4_18}
    \widetilde M=1+2\max\{\overline C, \max_{\overline\Omega_2}\{|\widetilde\Xi|, |\widetilde\Theta|\} \}.
\end{equation}
We aim to establish that
\begin{equation}\label{Sol4_19}
\max_{\Omega_{1}} \{|\widetilde\Xi(z,t)|, |\widetilde\Theta(z,t)|\} < \widetilde M,
\end{equation}
which implies that $\widetilde\Xi$ and $\widetilde\Theta$ remain uniformly bounded throughout the full domain $\Omega$. From the definition of $\widetilde M$ in (\ref{Sol4_18}), we have
\begin{equation}\label{Sol4_20}
    |\widetilde\Xi|, |\widetilde\Theta|<\frac{\widetilde M}{2}, \; \; \forall (z,t)\in (\widehat{M'P'}\cap \{t\le \bar{\epsilon}\})\cup (\Omega \cap \{t=\bar\epsilon\}).
\end{equation}
\begin{figure}
    \centering
    \includegraphics[width=0.45\linewidth]{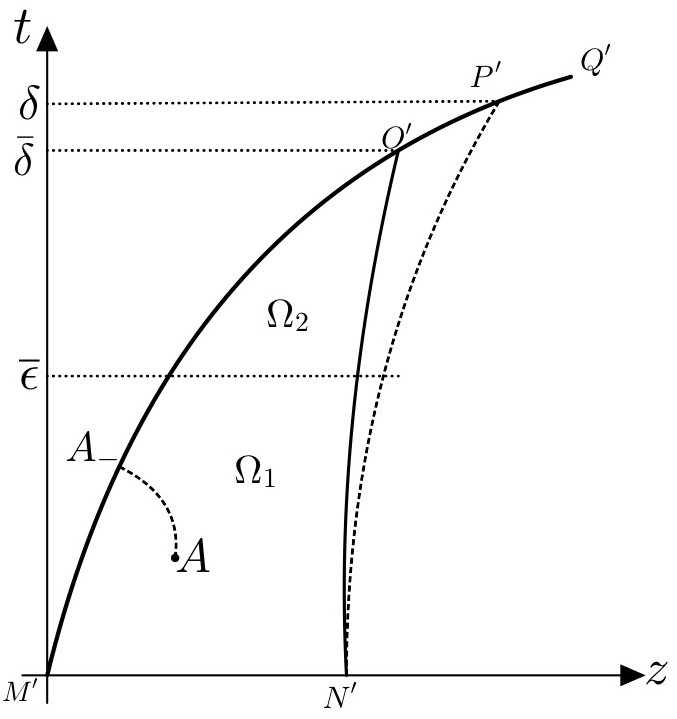}
    \caption{The region $M'N'O'$}
    \label{fig2}
\end{figure}
We prove this by a contradiction argument (\ref{Sol4_19}). Slide the level set of $t$ downward from $t=\bar{\epsilon}$ to $t=0$ and assume that there exists a first point $A(z,t)\in \Omega_{1}$ at which the bound is violated. Without loss of generality, suppose that $|\widetilde\Xi(A)| = \widetilde M$.

From $A$ trace the negative characteristic backward until it intersects either the boundary segment $\widehat{M'P'}\cap \{t\le \bar{\epsilon}\}$ or the horizontal line $t=\bar{\epsilon}$ at a point $A_{-}(z_{-}, t_{-})$(See Figure \ref{fig2} for more information). Along this negative characteristic from $A$ to $A_{-}$ we still have $|\widetilde\Xi|\le \widetilde M$ and $|\widetilde\Theta|\le \widetilde M$.

Integrating the $\widetilde\Xi$ in (\ref{Sol4_16}) along this characteristic from $A_{-}$ to $A$ yields
\begin{equation*}\label{Sol4_21}
{t_-}^{-\frac{1}{2}-\nu}\widetilde\Xi(A_-)-t^{-\frac{1}{2}-\nu}\widetilde\Xi(A)=\int_t^{t_-}s^{-\frac{3}{2}-\nu}\left(\hat f_1s\widetilde\Xi+\frac{\widetilde \Theta}{2}\hat f_2+\frac{s^{\nu}}{2}\hat f_{3}\right)ds.
\end{equation*}
From this relation, together with the selection of $\bar{\epsilon}$ in (\ref{Sol4_17}), we obtain
\begin{equation*}\label{Sol4_22}
    \begin{split}
        \widetilde M&=|\widetilde\Xi(A)|\\
                &=t^{\frac{1}{2}+\nu}\left|{t_-}^{-\frac{1}{2}-\nu}\widetilde\Xi(A_-)-\int_t^{t_-}s^{-\frac{3}{2}-\nu}\left(\hat f_1s\widetilde\Xi+\frac{\widetilde \Theta}{2}\hat f_2+\frac{s^{\nu}}{2}\hat f_{3}\right)ds\right|\\
                &\leq t^{\frac{1}{2}+\nu}{t_-}^{-\frac{1}{2}-\nu}|\widetilde\Xi(A_-)|+t^{\frac{1}{2}+\nu}\int_t^{t_-}s^{-\frac{3}{2}-\nu}\left(|\hat f_1|s|\widetilde\Xi|+\frac{|\widetilde \Theta|}{2}|\hat f_2|+\frac{s^{\nu}}{2}|\hat f_{3}|\right)ds\\
                &\leq \frac{\widetilde M}{2}t^{\frac{1}{2}+\nu}{t_-}^{-\frac{1}{2}-\nu}+t^{\frac{1}{2}+\nu}\int_t^{t_-}s^{-\frac{3}{2}-\nu}\left(\hat K \widetilde M \bar\epsilon+\frac{\widetilde M}{2}\hat K+\frac{{\bar\epsilon}^{\nu}}{2}\hat K\right)ds\\
                &\leq \frac{\widetilde M}{2}t^{\frac{1}{2}+\nu}{t_-}^{-\frac{1}{2}-\nu}+t^{\frac{1}{2}+\nu}\int_t^{t_-}s^{-\frac{3}{2}-\nu}\left(\frac{\nu}{4} \widetilde M+\frac{\widetilde M}{2}\hat K+\frac{\nu}{4}\widetilde M\right)ds\\
                &\leq \left( \frac{1}{2}-\frac{\nu+\hat K}{1+2\nu}\right)\widetilde M t^{\frac{1}{2}+\nu}{t_-}^{-\frac{1}{2}-\nu}+ \frac{\nu+\hat K}{1+2\nu} \widetilde M <\widetilde M.
    \end{split}
\end{equation*}
This yields a contradiction. In this step we have used the relations $t_-\le \bar{\epsilon}$ and $|\Xi(A_{-})|\le \frac{1}{2}\widetilde M$ from (\ref{Sol4_20}). If the first point $A(z,t)\in \Omega_{1}$ where a violation occurs satisfies $|\tilde \Theta(P)|=\widetilde M$, an analogous argument produces the same contradiction. Hence for every point $(z,t)\in \Omega_{1}$ the relation (\ref{Sol4_19}) holds.
This completes the proof of the lemma.
\end{proof}
\begin{lemma}\label{Sol_3_l4}
    The function $W$ remains uniformly bounded all the way up to the degenerate line $\widehat{M'N'}$.
\end{lemma}
\begin{proof}
    Define
    \begin{equation}\label{Sol4_23}
        \widetilde W=\frac{\widetilde X-\widetilde Y}{2t}=\frac{W}{ X Y},
    \end{equation}
    which implies together with \eqref{l11}, it suffices to prove that $\widetilde W$ is uniformly bounded. Using relations \eqref{Sol3_5} and \eqref{Sol4_14}, we can directly derive the governing equation satisfied by $\widetilde W$
    \begin{equation}\label{Sol4_24}
      \begin{split}
          \begin{cases}
               \tilde \partial_+ \widetilde W&=(H_{11}-H_{21})\widetilde W+\frac{1}{2}(H_{12}-H_{22})-\frac{1}{2}t^{1-\nu}\widetilde \Theta,\\
               \tilde \partial_- \widetilde W&=(H_{11}-H_{21})\widetilde W+\frac{1}{2}(H_{12}-H_{22})-\frac{1}{2}t^{1-\nu}\widetilde \Xi.
          \end{cases}
      \end{split}
    \end{equation}
   Note from (\ref{Sol1_3}) and (\ref{Sol1_4}) that the boundary values of $\widetilde W$ on $\widehat{M'P'}$ are given by 
   \begin{equation*}\label{Sol4_25}
       \widetilde W=\frac{\hat d}{\hat a\hat b}(z),
   \end{equation*}
   which is uniformly bounded. The uniform boundedness of $\widetilde W$ then follows by integrating the equation satisfied by $\widetilde W$ along the negative characteristic curve from an arbitrary point $(z,t)\in \Omega$ to the boundary $\widehat{M'P'}$, together with the already established uniform bounds on $H_{ij}$ and $(\widetilde\Xi, \widetilde\Theta)$. This completes the proof of the lemma.
\end{proof}
\begin{corollary}\label{Sol_3_c1}
The functions $X$ and $Y$ coincide along the degenerate curve $\widehat{M'N'}$. $X$ and $Y$ converge to this common limit on $\widehat{M'N'}$ at least at a linear rate with respect to $t$.
\end{corollary}
Hence, based on the solution properties stated in Corollary \ref{Sol_3_c1}, we are able to establish uniform regularity for $( X, Y, W)$.
\begin{lemma}\label{Sol_3_l5}
    The functions $( X, Y,W)$ are uniformly $C^{1/3}$ continuous throughout the entire domain $\Omega$, including the degenerate line $\widehat{M'N'}$.
\end{lemma}
\begin{proof}
    Let $A_{1}(z_{1},0)$ and $A_{2}(z_{2},0)$, with $z_{1}<z_{2}$, be two arbitrary points on the degenerate curve $\widehat{M'N'}$. From the point $(z_{1},0)$ we draw the positive characteristic and similarly draw the positive characteristic issuing from $(z_{2},0)$. Let us denote the intersection point of these two characteristic curves by $A_{m}(z_{m},t_{m})$. Recalling the expressions of $\tilde{\Lambda}_{\pm}$ given in (\ref{Sol3_3}), we can derive the corresponding relations satisfied by $t_{m}$ and $z_{m}$ as
    \begin{equation*}\label{Sol4_26}
        z_m=z_1-\int_0^{t_m} \frac{cft^2\widetilde X}{1-tg\widetilde X} dt=z_2+\int_0^{t_m} \frac{cft^2\widetilde Y}{1+tg\widetilde Y}dt,
    \end{equation*}
    which combined with (\ref{l11}), implies that
    \begin{equation}\label{Sol4_27}
        K_1t_m\leq |z_2-z_1|^{1/3}\leq K_2t_m,
    \end{equation}
    for some uniform positive constants $K_1$ and $K_2$. Since $\widetilde W$ is uniformly bounded, we may then apply (\ref{Sol3_5}) together with Lemma \ref{Sol_3_l3} to obtain
    \begin{equation}\label{Sol4_28}
        |\tilde\partial_+\widetilde X|, |\tilde\partial_-\widetilde Y|, |\widetilde X_t|, |\widetilde Y_t|, t^{1/2}|\widetilde X_z|, t^{1/2}|\widetilde Y_z|\leq {a_1},
    \end{equation}
    for some positive constant $a_{1}>0$. Therefore, by Corollary \ref{Sol_3_c1} together with (\ref{Sol4_27}) and (\ref{Sol4_28}), we obtain
    \begin{equation}\label{Sol4_29}
        \begin{split}
            |\widetilde X(z_1,0)-\widetilde X(z_2,0)|&=|\widetilde Y(z_1,0)-\widetilde X(z_2,0)|\\
            &\leq |\widetilde Y(z_1,0)-\widetilde Y(z_m,t_m)|+|\widetilde Y(z_m,t_m)-\widetilde X(z_m,t_m)|+|\widetilde X(z_m,t_m)-\widetilde X(z_2,0)|\\
            &\leq a_1t_m+2\sup_{\Omega}|\widetilde W|\cdot t_m+ a_1t_m\\
            &\leq \frac{2}{K_2}(a_1+\sup_{\Omega}|\widetilde W|)|z_2-z_1|^{1/3}.
        \end{split}
    \end{equation}
    Now consider any two points $A_{3}(z_{1},t_{1})$ and $A_{4}(z_{2},t_{2})$ with $z_{1}\le z_{2}$ and $0\le t_{1}\le t_{2}$ in the domain $\Omega$. If $t_{1}\ge (z_{2}-z_{1})$, then by (\ref{Sol4_28})
   \begin{equation}\label{Sol4_30}
        \begin{split}
            |\widetilde X(z_1,t_1)-\widetilde X(z_2,t_2)|&=|\widetilde X(z_1,t_1)-\widetilde X(z_2,t_1)|+|\widetilde X(z_2,t_1)-\widetilde X(z_2,t_2)| \\
            &\leq \sup_{\Omega}|\widetilde X_z|\cdot |z_1-z_2|+\sup_{\Omega}|\widetilde X_t|\cdot|t_1-t_2|\\
            & \leq\frac{a_1}{t^{1/2}}|z_1-z_2|+a_1|t_1-t_2|\\
            &\leq a_1|z_1-z_2|^{1/2}+a_1|t_1-t_2|\\
            &\leq 2a_1|(z_1,t_1)-(z_2,t_2)|^{1/3},
        \end{split}
    \end{equation}
    and if $t_{1}<(z_{2}-z_{1})$, then by (\ref{Sol4_28}) and (\ref{Sol4_29}) we again obtain
     \begin{equation}\label{Sol4_31}
        \begin{split}
            |\widetilde X(z_1,t_1)-\widetilde X(z_2,t_2)|=&|\widetilde X(z_1,t_1)-\widetilde X(z_1,0)|+|\widetilde X(z_1,0)-\widetilde X(z_2,0)|+|\widetilde X(z_2,0)-\widetilde X(z_2,t_1)|\\ &+|\widetilde X(z_2,t_1)-\widetilde X(z_2,t_2)|\\
           \leq  &\sup_{\Omega}|\widetilde X_t|\cdot t_1+ \frac{2}{K_2}(a_1+\sup_{\Omega}|\widetilde W|)|z_2-z_1|^{1/3}+               \sup_{\Omega}|\widetilde X_t|\cdot t_1+\sup_{\Omega}|\widetilde X_t|\cdot |t_1-t_2|\\
            \leq  &\left( 3a_1+ \frac{2}{K_2}(a_1+\sup_{\Omega}|\widetilde W|)\right)|(z_1,t_1)-(z_2,t_2)|^{1/3}.
        \end{split}
    \end{equation}
    Combining (\ref{Sol4_30}) and (\ref{Sol4_31}), we obtain the uniform $C^{1/3}$ continuity of $\widetilde X$ throughout the entire domain $\Omega$. An analogous argument establishes the same uniform regularity for $\widetilde Y$. The function $\widetilde W$ can be treated similarly by using (\ref{Sol4_23}). Finally, the conclusion of the lemma follows from the relations between $( X, Y,W)$ and $(\widetilde X,\widetilde Y,\widetilde W)$ together with Lemma \ref{Sol_3_l1}.
\end{proof}
As a consequence of Lemma \ref{Sol_3_l5}, the uniform regularity of $( X, Y,W)$ can be further enhanced.
\begin{lemma}\label{Sol_3_l6}
    The functions $( X, Y)(z,t)$ are uniformly $C^{1-\nu}$ continuous and $W(z,t)$ is uniformly $C^{(2-\nu)/3}$ continuous for any $\nu\in(0,1)$ throughout the domain $\Omega$ including the degenerate line $\widehat{M'N'}$.
\end{lemma}
 \begin{proof}
     It is sufficient to prove that $(\widetilde X, \widetilde Y)$ are uniformly $C^{1-\nu}$ continuous and that $W$ is uniformly $C^{(2-\nu)/3}$ continuous for any $\nu\in(0,1)$. Let $M^{\ast}$ denote a uniform positive constant such that, for $i,j=1,2$,
     \begin{equation}\label{Sol4_32}
         \left | \frac{cf\widetilde X}{1-tg\widetilde X}\right|, \left | \frac{cf\widetilde Y}{1+tg\widetilde Y}\right|, |2tH_{ij}+1|, |\widetilde W|, |H_{ij\widetilde X}|, |H_{ij\widetilde Y}|,|H_{ij\widetilde Y}c_z|\leq M^*.
     \end{equation}
     The inequalities above follow from Lemma \ref{Sol_3_l1} together with the explicit formulas for $f$, $g$ and $H_{ij}$ given in (\ref{Sol3_6}) and (\ref{Sol1_10}). Furthermore, by Lemma \ref{Sol_3_l3}, there exists a uniform positive constant $a_{2}$, depending on the parameter $\nu$, such that
     \begin{equation}\label{Sol4_33}
         |\widetilde\Xi|, |\widetilde\Theta|, t^\nu|\widetilde X_z|, t^\nu|\widetilde Y_z|\leq a_2.
     \end{equation}
     We begin by refining the uniform regularity of $\widetilde W$. Recalling the first equation satisfied by $\widetilde W$ in (\ref{Sol4_24}), we obtain
     \begin{equation}\label{Sol4_34}
         \widetilde W_t=-\frac{cft\widetilde Y}{2(1+tg\widetilde Y)}(\widetilde X_z-\widetilde Y_z) +(H_{11}-H_{21})\widetilde W+\frac{1}{2}(H_{12}-H_{22})-\frac{1}{2}t^{1-\nu}\widetilde \Theta.
     \end{equation}
     Integrating (\ref{Sol4_34}) along the segment from $(z_{i},0)$ to $(z_{i},t_{m})$ for $i=1,2$ yields
     \begin{equation}\label{Sol4_35}
          \widetilde W(z_i,0)=\widetilde W(z_i,t_m)-\int_0^t\left(-\frac{cft\widetilde Y}{2(1+tg\widetilde Y)}(\widetilde X_z-\widetilde Y_z) +(H_{11}-H_{21})\widetilde W+\frac{1}{2}(H_{12}-H_{22})-\frac{1}{2}t^{1-\nu}\widetilde \Theta \right)dt.
     \end{equation}
     By using (\ref{Sol4_35}) we obtain
     \begin{equation}\label{Sol4_36}
          \begin{split}
              |\widetilde W(z_1,0)-\widetilde W(z_2,0)|&\leq |\widetilde W(z_1,t_m)-\widetilde W(z_2,t_m)| \\&+\int_0^{t_m}\left(\left|\frac{cf\widetilde Y (\widetilde X_z t^\nu-\widetilde Y_z t^\nu)}{2(1+tg\widetilde Y)}(z_1,t) \right|+\left|\frac{cf\widetilde Y (\widetilde X_z t^\nu-\widetilde Y_z t^\nu)}{2(1+tg\widetilde Y)}(z_2,t) \right|\right)t^{1-\nu}dt\\
              &+\int_0^{t_m} |(H_{11}-H_{21})(z_1,t)|\cdot |\widetilde W(z_1,t)-\widetilde W(z_2,t)| dt \\
              &+ \int_0^{t_m}  |(H_{11}-H_{21})(z_1,t)-(H_{11}-H_{21})(z_2,t)|\cdot |\widetilde W(z_2,t)|dt\\
              &+\int_0^{t_m} \frac{1}{2}|(H_{12}-H_{22})(z_1,t)-(H_{12}-H_{22})(z_2,t)|dt+\int_0^{t_m}\frac{1}{2}\left(|\widetilde \Theta(z_1,t)|+|\widetilde \Theta(z_2,t)| \right)t^{1-\nu} dt.
          \end{split}
     \end{equation}
     Consider the leading term on the right hand side of (\ref{Sol4_36}). Using the estimates provided in (\ref{Sol4_33})  and (\ref{Sol4_27}), we deduce that
     \begin{equation}\label{Sol4_37}
         \begin{split}
             |\widetilde W(z_1,t_m)-\widetilde W(z_2,t_m)|&\leq \frac{1}{t_m}\left( \left|\widetilde X(z_1,t_m)-\widetilde X(z_2,t_m)\right|+\left|\widetilde Y(z_1,t_m)-\widetilde Y(z_2,t_m)\right|\right)\\
             &\leq  t_{m}^{-1-\nu}\left(\left|t_m^\nu\widetilde X(z',t_m)\right|\cdot|z_1-z_2|+ \left|t_m^\nu\widetilde Y(z'',t_m)\right|\cdot|z_1-z_2|\right)\\
             &\leq 2a_2t_{m}^{-1-\nu}|z_1-z_2|\leq 2a_2 \;K_2^{1+\nu}|z_1-z_2|^\frac{2-\nu}{3}.
         \end{split}
     \end{equation}
     Furthermore, since the functions $(\widetilde X,\widetilde Y,\widetilde W)(z,t)$ possess uniform $C^{1/3}$ continuity as established in Lemma \ref{Sol_3_l5}, it follows that
     \begin{equation}\label{Sol4_38}
       |\widetilde X(z_1,t)-\widetilde X(z_2,t)|, |\widetilde Y(z_1,t)-\widetilde Y(z_2,t)|, |\widetilde W(z_1,t)-\widetilde W(z_2,t)|\leq a_3|z_1-z_2|^{1/3},
     \end{equation}
     for a certain uniform positive constant $a_3$. Consequently, combining (\ref{Sol4_32}) and (\ref{Sol4_38}), we conclude that for $i=1,2$,
     \begin{equation*}\label{Sol4_39}
         \begin{split}
             \left| (H_{1i}-H_{2i})(z_1,t)-(H_{1i}-H_{2i})(z_2,t)\right| &\leq |H_{1i\widetilde X}-H_{2i\widetilde X}|\cdot|\widetilde X(z_1,t)-\widetilde X(z_2,t)|+ |H_{1i\widetilde Y}-H_{2i\widetilde Y}|\\& \cdot|\widetilde Y(z_1,t)-\widetilde Y(z_2,t)|+|H_{1ic}-H_{2ic}|\cdot c_z\cdot |z_1-z_2|\\
             &\leq 4M^*a_3|z_1-z_2|^{1/3}+2M^*|z_1-z_2|.
         \end{split}
     \end{equation*}
     We now substitute (\ref{Sol4_37}) and (\ref{Sol4_38}) into (\ref{Sol4_36}) and then apply the bounds given in  (\ref{Sol4_27}), (\ref{Sol4_32}) and (\ref{Sol4_33}) to obtain
     \begin{equation}\label{Sol4_40}
         \begin{split}
             |\widetilde W(z_1,0)-\widetilde W(z_2,0)|\leq &2a_2 \;K_2^{1+\nu}|z_1-z_2|^\frac{2-\nu}{3}+\int_0^{t_m} a_2(2M^*+1)t^{1-\nu} dt\\
             &+ \int_0^{t_m} 2M^* a_3|z_1-z_2|^{1/3} dt+\int_0^{t_m} \left(M^*+\frac{1}{2}\right)\left(4M^*a_3|z_1-z_2|^{1/3}+2M^*|z_1-z_2|\right)dt\\
             \leq & 2a_2 \;K_2^{1+\nu}|z_1-z_2|^\frac{2-\nu}{3}+\frac{2M^*+1}{2-\nu}a_2t_m^{2-
             \nu}+2M^*a_3|z_1-z_2|^{1/3}t_m\\
             &+ (M^*+1)\left(4M^*a_3|z_1-z_2|^{1/3}+2M^*|z_1-z_2| \right)t_m\\
             \leq & a_4 |z_1-z_2|^{\frac{2-\nu}{3}},
         \end{split}
     \end{equation}
     for a uniform positive constant $a_4$ that depends on $\nu$. In view of (\ref{Sol4_40}), the same argument employed in Lemma \ref{Sol_3_l5} can be repeated to conclude that the function $\widetilde W$ is uniformly $C^{(2-\nu)/3}$ continuous throughout the entire domain $\Omega$.

     For the functions $\widetilde X$ and $\widetilde Y$, we return to (\ref{Sol3_5}) to obtain
     \begin{equation}\label{Sol4_41}
         \begin{split}
             \begin{cases}
                 \widetilde X_t =-\frac{cft^2\widetilde Y}{1+tg\widetilde Y}\widetilde X_z+(2tH_{11}+1)\widetilde W+H_{12}t,\\
                 \widetilde Y_t =\frac{cft^2\widetilde X}{1-tg\widetilde X}\widetilde Y+(2tH_{21}-1)\widetilde W+H_{22}t.
             \end{cases}
         \end{split}
     \end{equation}
     Integrating the equation satisfied by $\widetilde X$ in (\ref{Sol4_41}) along the path from $(z_{i},0)$ to $(z_{i},t_{m})$ for $i=1,2$, we obtain
     \begin{equation*}\label{Sol4_42}
         \widetilde X(z_i,0)=\widetilde X(z_i,t_m)-\int_0^{t_m} \left(-\frac{cft^2\widetilde Y}{1+tg\widetilde Y}+(2tH_{11}+1)\widetilde W+H_{12}t\right)(z_i,t)dt,
     \end{equation*}
     from which it follows that
     \begin{equation}\label{Sol4_43}
         \begin{split}
             |\widetilde X(z_1,0)-\widetilde X(z_2,0)|\leq & |\widetilde X(z_1,t_m)-\widetilde X(z_2,t_m)|+\int_0^{t_m} |H_{12}(z_1,t)-H_{12}(z_2,t)|tdt\\
              +\int_0^{t_m}&\left( \left|\frac{cf\widetilde Y}{1+tg\widetilde Y}(z_1,t)\right| \cdot|t^\nu\widetilde X_z(z_1,t)|+ \left|\frac{cf\widetilde Y}{1+tg\widetilde Y}(z_2,t)\right| \cdot|t^\nu\widetilde X_z(z_2,t)|\right)t^{2-\nu}dt\\
              +\int_0^{t_m}&\left( |2tH_{11}(z_1,t)+1|\cdot|\widetilde W(z_1,t)-\widetilde W(z_2,t)|+2t|\widetilde W|+2t|\widetilde W|\cdot|H_{11}(z_1,t)-H_{11}(z_2,t)| \right)dt.
         \end{split}
     \end{equation}
     By making use of the bounds in (\ref{Sol4_32}) and (\ref{Sol4_33}), we conclude that
     \begin{equation}\label{Sol4_44}
         \begin{split}
             |\widetilde X(z_1,t_m)-\widetilde X(z_2,t_m)|&=|\widetilde X_z(z',t_m)|\cdot|z_1-z_2|\leq a_2t_m^{-\nu} |z_1-z_2|\leq a_2{K_2}^\nu |z_1-z_2|^{1-\nu}.
         \end{split}
     \end{equation}
     Also, we have
     \begin{equation*}\label{Sol4_45}
         \left|\frac{cf\widetilde Y}{1+tg\widetilde Y}(z_1,t)\right| \cdot|t^\nu\widetilde X_z(z_1,t)|+ \left|\frac{cf\widetilde Y}{1+tg\widetilde Y}(z_2,t)\right| \cdot|t^\nu\widetilde X_z(z_2,t)|\leq 2M^*a_2.
     \end{equation*}
     Moreover, (\ref{Sol4_32}) implies that for $i=1,2$,
     \begin{equation*}\label{Sol4_46}
         \begin{split}
             |H_{1i}(z_1,t)-H_{1i}(z_2,t)|\leq& |H_{1i\widetilde X}|\cdot |\widetilde X(z_1,t)-\widetilde X(z_2,t)|+|H_{1i\widetilde Y}|\cdot |\widetilde Y(z_1,t)-\widetilde Y(z_2,t)|\\
             &+|H_{1i\tilde c}|\cdot |\tilde c(z_1,t)-\tilde c(z_2,t)|\\
             \leq &\left(2a_2t^{-\nu}+1 \right)M^*|z_1-z_2|.
         \end{split}
     \end{equation*}
     As a consequence of the uniform $C^{(2-\nu)/3}$ continuity of $\widetilde W$, it follows that
     \begin{equation}\label{Sol4_47}
         |\widetilde W(z_1,t)-\widetilde W(z_2,t)|\leq a_5|z_1-z_2|^{\frac{2-\nu}{3}},
     \end{equation}
     for a uniform positive constant $a_5$ depending on $\nu$. Employing \eqref{Sol4_44}-\eqref{Sol4_47} into \eqref{Sol4_43} and then using \eqref{Sol4_27} and \eqref{Sol4_32}, we conclude that
     \begin{equation}\label{Sol4_48}
         \begin{split}
              |\widetilde X(z_1,0)-\widetilde X(z_2,0)|&\leq a_2K_2^\nu|z_1-z_2|^{1-\nu}+2M^*a_2\int_0^{t_m} t^{2-\nu} dt\\
              &+M^*\int_0^{t_m} \left(a_5|z_1-z_2|^{\frac{2-\nu}{3}}+t(2M^*+1)(2a_2t^{-\nu}+1)|z_1-z_2| \right)dt\\
              &= a_2K_2^\nu|z_1-z_2|^{1-\nu}+\frac{2M^*a_2}{3-\nu}t_m^{3-\nu} +M^*a_5|z_1-z_2|^{\frac{2-\nu}{3}}t_m+(2M^*+1)(2a_2+1)M^*|z_1-z_2|\\&\leq a_6|z_1-z_2|^{1-\nu},
         \end{split}
     \end{equation}
     for a uniform positive constant $a_6$. Based on (\ref{Sol4_48}), arguments analogous to those used in Lemma \ref{Sol_3_l5} can be applied to conclude that the function $\widetilde X$ is uniformly $C^{1-\nu}$ continuous over the entire domain $\Omega$. By the same reasoning, the function $\widetilde Y$ also enjoys uniform $C^{1-\nu}$ continuity. This completes the proof of the lemma.
 \end{proof}
 Finally, we trace the positive characteristic issuing from the point $N'(\bar z(0),0)$ until it reaches the boundary $\widehat{M'P'}$ and denote the intersection point by $O'(\tilde z(t_{o}), t_{o})$. Choosing $\bar{\delta}=t_{o}$ completes the proof of Theorem \ref{Sol3_th1}.

%%%%%%%%%%%%%%%%%%%%%%%%%%%%%%%%%%%%%%%%%%%%%%%%%%%%%%%%%%%%%%%%%%%%%%%%%%%%%%%%%%%%%%%%%%%%%%%%%%%%%%%%%%%

\section{Regularity of solution in the self-similar plane}
Relying on the solution of the degenerate problem (\ref{Sol1_3}) and (\ref{Sol1_8}) formulated in the preceding section, we construct a regular supersonic solution to the boundary value problem (\ref{Pre2_3}) and (\ref{Pre3_2}) in the self-similar $(\xi,\eta)$ plane. This completes the proof of Theorem \ref{Pre_th1} in the present section.
     
%%%%%%%%%%%%%%%%%%%%%%%%%%%%%%%%%%%%%%%%%%%%%%%%%%%%%%
\subsection{Inversion}

As a consequence of Theorem \ref{Sol3_th1}, the functions $(X,Y)(z,t)$ are obtained and defined throughout the entire region $M'N'O'$. In order to recover a solution in the self-similar $(\xi,\eta)$ plane, it is necessary to introduce the coordinate mappings $\xi=\xi(z,t)$ and $\eta=\eta(z,t)$ and to analyze their invertibility. Recalling the coordinate transformation (\ref{Sol1_1}) and applying (\ref{Pre2_11}) and (\ref{Pre2_12}), we obtain
\begin{equation}\label{Inv1_4}
    \xi_t=-\frac{c\sin\theta t}{(1-t^2)J},\; \eta_t=\frac{c\cos\theta t}{(1-t^2)J},\; \xi_z=\frac{\bar\omega_\eta}{J},\;\eta_z=-\frac{\bar\omega_\xi}{J}
\end{equation}
where
\begin{equation}\label{Inv1_5}
    \begin{split}
       \bar\omega_\xi &=\cos\theta\sqrt{1-t^2} (1+\kappa(\tau)(1-t^2)){W}+\frac{\cos\theta(1-t^2)}{c}-\sin\theta\frac{(1+\kappa(\tau)(1-t^2))}{2}(  X- Y),\\
       \bar \omega_\eta &=\sin\theta\sqrt{1-t^2}(1+\kappa(\tau)(1-t^2)){W}+\frac{\sin\theta(1-t^2)}{c}+\cos\theta\frac{(1+\kappa(\tau)(1-t^2))}{2}(  X- Y),\\
        J &=\phi_\xi\bar\omega_n-\phi_\eta\bar\omega_\xi=-\frac{c({ X}-{ Y})}{2t}(1+\kappa(\tau)(1-t^2)).
    \end{split}
\end{equation}
This observation implies that
\begin{equation}\label{Inv1_6}
    \partial_-\xi=-\frac{t^2f(\sqrt{1-t^2}\sin\theta +ct\cos\theta)}{Y-tg}, \; \partial_-\eta=\frac{t^2f(\sqrt{1-t^2}\cos\theta -ct\sin\theta)}{Y-tg}.
\end{equation}
For an arbitrary point $(z^0, t^0)$ in the region $M'N'O'$, we therefore obtain, by integrating (\ref{Inv1_5}) along the negative characteristic, that
\begin{equation*}\label{Inv1_7}
    \begin{split}
        \xi(z^0,t^0) &= \xi^*(\tilde z(t^*))+\int_{t^0}^{t^*}\frac{t^2f(\sqrt{1-t^2}\sin\theta +ct\cos\theta)}{Y-tg}(z_{-}(t; t^0, z^0),t)dt,\\
        \eta(z^0,t^0) &= \psi(\xi^*(\tilde z(t^*)))-\int_{t^0}^{t^*}  \frac{t^2f(\sqrt{1-t^2}\cos\theta -ct\sin\theta)}{Y-tg} (z_{-}(t; t^0, z^0),t)dt.
    \end{split}
\end{equation*}
Since $(z^0, t^0)$ is arbitrary, the relations in (\ref{Inv1_6}) define two functions $\xi=\xi(z,t)$ and $\eta=\eta(z,t)$ throughout the entire region $M'N'O'$.

Further, we construct the function $\theta(z,t)$. Making use of (\ref{Pre2_11}) and (\ref{Sol1_6}), we derive
\begin{equation}\label{Inv1_1}
    \partial_-\theta = -\frac{ft^2\sqrt{1-t^2}}{ Y-gt}\left(-\kappa(\tau)t\sqrt{1-t^2} Y
        +\frac{\bar\omega^2}{c}\right).
\end{equation}
For an arbitrary point $(z_0,t_0)$ in the region $M'N'O'$, we trace the negative characteristic curve $z=z_{-}(t; t_0, z_0)$ for $t\ge t_0$ until it intersects the boundary $M'O'$ at a unique point $(\tilde z(t^{\ast}), t^{\ast})$ satisfying
\begin{equation*}\label{Inv1_2}
    \begin{split}
        \begin{cases}
            \frac{dz_-}{dt}(t; t_0, z_0)&=-\frac{cft^2}{ Y-tg}(z_-(t; t_0, z_0),t),\\
            z_-(t_0; t_0, z_0)&=z_0,\\
            z_-(t^*; t_0, z_0)&=\tilde z(t^*).
        \end{cases}
    \end{split}
\end{equation*}
We integrate (\ref{Inv1_1}) along the negative characteristic $z=z_{-}(t; t_0, z_0)$, starting from the point $(z_0, t_0)$ and proceeding to the boundary point $(\tilde z(t^{\ast}), t^{\ast})$. By incorporating the prescribed boundary values of $\theta$ on $\widehat{M'O'}$, we obtain
\begin{equation*}\label{Inv1_3}
    \theta(z_0,t_0)=\theta^*(\xi^*(\tilde z(t^*)))+\int_{t_0}^{t^*}  \frac{ft^2\sqrt{1-t^2}}{ Y-gt}\left(-\kappa(\tau)t\sqrt{1-t^2} Y
        +\frac{\bar\omega^2}{c}\right)(z_{-}(t; t_0, z_0),t)dt.
\end{equation*}
As a result, the function $\theta(z,t)$ is well defined throughout the entire region $M'N'O'$. 
%Moreover, the mapping $(z,t)\mapsto(\xi,\eta)$ is globally one to one. This property follows directly from (\ref{Inv1_4}), which implies that
%\begin{equation}\label{Inv1_8}
 %  J =-\frac{c({ X}-{ Y})}{2t}(1+\kappa(\tau)(1-t^2))<0.
%\end{equation}
%The inequality above shows that $\phi$ is strictly decreasing along each level curve satisfying $(1-\omega)\ge 0$.

%%%%%%%%%%%%%%%%%%%%%%%%%%%%%%%%%%%%%%%%%%%%%%%%%%%%%%%%%
\subsection{Mapping $(z,t)\mapsto(\xi,\eta)$ is one to one}

We argue by contradiction. Suppose that there exist two distinct points $(\xi_{1},\eta_{1})$ and $(\xi_{2},\eta_{2})$ in the region $M'N'O'$ such that $t_{1}=t_{2}$ and $z_{1}=z_{2}$. This assumption implies that
$\cos\omega(\xi_{1},\eta_{1})=\cos\omega(\xi_{2},\eta_{2})$ and
$\phi(\xi_{1},\eta_{1})=\phi(\xi_{2},\eta_{2})$.
Consequently, both points $(\xi_{1},\eta_{1})$ and $(\xi_{2},\eta_{2})$ lie on the same level curve defined by
$l_{\vartheta}=1-\bar\omega(\xi,\eta)\ge 0$.

It then follows from (\ref{Inv1_4}) that
\begin{equation}\label{Inv1_8}
   \nabla\phi\cdot(\bar\omega_\eta, -\bar\omega_\xi) =-\frac{c({ X}-{ Y})}{2t}(1+\kappa(\tau)(1-t^2))<0.
\end{equation}
The inequality above shows that $\phi$ is strictly decreasing along each level curve satisfying $l_{\vartheta}=(1-\bar\omega)\ge 0$. This contradicts the assumption that $\phi(\xi_{1},\eta_{1})=\phi(\xi_{2},\eta_{2})$. Hence, the mapping $(z,t)\mapsto(\xi,\eta)$ is globally one to one.

%%%%%%%%%%%%%%%%%%%%%%%%%%%%%%%%%%%%%%%%%%%%%%%%%%%%%%%%
\subsection{Proof of Theorem \ref{Pre_th1}}

Due to the global one to one nature of the mapping $(z,t)\mapsto(\xi,\eta)$, we may uniquely recover the inverse functions $t=\hat t(\xi,\eta)$ and $z=\hat z(\xi,\eta)$. Moreover, it also follows that
\begin{equation}\label{main_1}
    \hat t_\xi =\frac{\eta_z}{j},\; \hat t_\eta =-\frac{\xi_z}{j},\;
     \hat z_\xi =-\frac{\eta_t}{j},\; \hat z_\eta =\frac{\xi_t}{j},
\end{equation}
where $j=\xi_{t}\eta_{z}-\eta_{t}\xi_{z}$ and the pairs $(\xi_{t},\xi_{z})$ and $(\eta_{t},\eta_{z})$ are specified in (\ref{Inv1_4}). We now introduce the functions $(c,\theta,\bar\omega)$ expressed in the variables $(\xi,\eta)$ as follows
\begin{equation}\label{main_2}
    c=c(\hat z(\xi,\eta), \hat t(\xi,\eta)), \; \theta=\theta(\hat z(\xi,\eta), \hat t(\xi,\eta)),\; \bar\omega =\sqrt{1-\hat t^2(\xi,\eta)},\; \forall\; (\xi,\eta)\in \overline{MNO},
\end{equation}
where the region $MNO$ is enclosed by the curves $\widehat{MN}$, $\widehat{NO}$ and $\widehat{MO}$. In particular, the curves $\widehat{MN}$ and $\widehat{NO}$ are characterized by
\begin{equation}\label{main_4}
    \begin{split}
        \widehat{MN}&=\{(\xi,\eta)\;| \; \bar\omega(\xi,\eta)=1, \;\xi\in[\xi',\xi_2]\},\\
        \widehat{NO}&=\{(\xi,\eta)\;| \; \hat z(\xi,\eta)=z_+(\hat t(\xi,\eta); \bar\delta, \tilde z(\bar\delta)), \; \xi\in[\xi'',\xi']\},
    \end{split}
\end{equation}
where the point $\xi'=\xi(\bar z(0),0)$ and the function $z_{+}(t;\bar{\delta},\tilde z(\bar{\delta}))$ denotes the solution to the following ordinary differential equation
\begin{equation}\label{main_5}
    \begin{split}
        \begin{cases}
            \frac{dz_+}{dt}(t; \bar\delta, \tilde z(\bar{\delta}))&=\frac{cft^2 X}{ X+tg}(z_+(t; \bar\delta, \tilde z(\bar{\delta})),t),\;\;\;\;\; \;\;t\in[0,\bar\delta],\\
        z_+(\bar\delta; \bar\delta, \tilde z(\bar{\delta})) &=\tilde z(\bar\delta).
        \end{cases}
    \end{split}
\end{equation}
Also, $\xi''$ satisfies
\begin{equation*}\label{main_6}
    \hat z(\xi'',\psi(\xi''))=z_+(\tilde t(\xi'',\psi(\xi'')); \bar\delta,\tilde z(\bar\delta)).
\end{equation*}
In the self-similar plane, the points $N$ and $O$ have coordinates $(\xi',\eta(\bar z(0),0))$ and $(\xi'',\psi(\xi''))$, respectively. By the construction of the coordinate functions $(\xi(z,t),\eta(z,t))$, it follows that the functions $(c(\xi,\eta),\theta(\xi,\eta),\bar\omega(\xi,\eta))$ defined in  (\ref{main_2}) satisfy the boundary conditions prescribed in (\ref{Pre3_2}). Based on (\ref{main_2}), we further introduce the functions $(\alpha, \omega,\beta)(\xi,\eta)$ for all $(\xi,\eta) \in \overline{MNO}$ as follows
\begin{equation}\label{main_7}
    \alpha(\xi,\eta)=\theta(\xi,\eta)+\omega(\xi,\eta),\;
    \omega(\xi,\eta)=\arcsin\bar\omega(\xi,\eta),\;
    \beta(\xi,\eta)=\theta(\xi,\eta)-\omega(\xi,\eta).
\end{equation}
To address the regularity of the functions $(c(\xi,\eta),\theta(\xi,\eta),\bar\omega(\xi,\eta))$ defined in (\ref{main_2}), we state the following result.
\begin{lemma}\label{Main_l1}
    The functions $(c(\xi,\eta),\theta(\xi,\eta),\bar\omega(\xi,\eta))$ introduced in (\ref{main_2}) are uniformly $C^{1,\mu}$ continuous for any $\mu\in(0,1/3)$ over the entire region $MNO$. In addition, the sonic curve $MN$ is also $C^{1,\mu}$ continuous.
\end{lemma}
\begin{proof}
   For the function $\theta(\xi,\eta)$, it follows from (\ref{Pre2_2}) and  (\ref{Pre2_11}) that
   \begin{equation}\label{main_8}
       \begin{split}
           \theta_\xi &= -\frac{\kappa(\tau)\cos\theta\sqrt{1-t^2}}{2}( X- Y)+\frac{\sin\theta}{\sqrt{1-t^2}}\left( \frac{1-t^2}{c}+\frac{1}{2}\kappa(\tau)\sqrt{1-t^2}\cos\omega( X+ Y)\right),\\
           \theta_\eta &=-\frac{\kappa(\tau)\sin\theta\sqrt{1-t^2}}{2}( X- Y)-\frac{\cos\theta}{\sqrt{1-t^2}}\left( \frac{1-t^2}{c}+\frac{1}{2}\kappa(\tau)\sqrt{1-t^2}\cos\omega( X+ Y)\right).
       \end{split}
   \end{equation}
   In view of (\ref{main_8}) and (\ref{Inv1_5}), it is sufficient to establish that the functions $( X, Y,W)(\xi,\eta)$ are uniformly $C^{\mu}$- continuous for any $\mu\in(0,1/3)$ throughout the region $MNO$. These regularity properties follow from Lemma \ref{Sol_3_l6} together with the following assertion. If a function $I(z,t)$ belongs to $C^{2\bar\mu}$ on the entire region $M'N'O'$, then the transformed function $\bar I(\xi,\eta):=I(\hat z(\xi,\eta),\hat t(\xi,\eta))$ is uniformly $C^{\bar\mu}$ continuous on $MNO$.

   To justify this claim, let $(\xi^*,\eta^{*})$ and $(\xi^{**},\eta^{**})$ be two arbitrary points in $MNO$ and let $(z^{*},t^{*})$ and $(z^{**},t^{**})$ denote their corresponding points in the $(z,t)$ plane. We then examine the differences
   \begin{equation}\label{main_9}
       |z^{*}-z^{**}|=|\phi(\xi^*,\eta^{*})-\phi(\xi^{**},\eta^{**})|\leq a_7(|\xi^{**}-\xi^{*}|+|\eta^{**}-\eta^{*}|),
   \end{equation}
   and
   \begin{equation}\label{main_10}
       \begin{split}
           |t^{*}-t^{**}|^2 &\leq |t^{*}-t^{**}||t^{*}+t^{**}|=|{t^{*}}^2-{t^{**}}^2|=|\bar{\omega}^2(\xi^{*},\eta^{*})-\bar{\omega}^2(\xi^{**},\eta^{**})|\\
           &\leq 2 |\bar{\omega}(\xi^{*},\eta^{*})-\bar{\omega}(\xi^{**},\eta^{**})|\leq a_8(|\xi^{**}-\xi^{*}|+|\eta^{**}-\eta^{*}|),
       \end{split}
   \end{equation}
   where
   \begin{equation*}\label{main_11}
      a_7=\max_{(\xi,\eta)\in MNO} \frac{c}{\bar\omega}(\xi,\eta), \; a_8=2\max\left\{\max_{(\xi,\eta)\in MNO} |\bar\omega_\xi|, \; \max_{(\xi,\eta)\in MNO} |\bar\omega_\eta|\right\}.
   \end{equation*}
   Here $a_7$ and $a_8$ denote two positive constants determined by (\ref{Inv1_5}) and Lemmas \ref{Sol_3_l1} and \ref{Sol_3_l4}. Combining these bounds with (\ref{main_9}) and (\ref{main_10}) and invoking the $C^{2\bar\mu}$ continuity, we obtain
   \begin{equation*}\label{main_12}
       \begin{split}
           |\tilde I(\xi^{*},\eta^{*})-\tilde I(\xi^{**},\eta^{**})|&=|I(z^{*},t^{*})-I(z^{**},t^{**})|\\
           &\leq a_9|(z^{*},t^{*})-(z^{**},t^{**})|^{2\bar\mu}\\
           &\leq a_9 \left(|z^{*}-z^{**}|^2+|t^{*}-t^{**}|^2 \right)^{\bar\mu}\\
           &\leq a_9|(\xi^{*},\eta^{*})-(\xi^{**},\eta^{**})|^{\bar\mu},
       \end{split}
   \end{equation*}
   for some uniform positive constant $a_9$, which implies that the function $\tilde I(\xi,\eta)$ is uniformly $C^{\bar\mu}$ continuous throughout the region $MNO$. By Lemma \ref{Sol_3_l6}, we further conclude that $( X(\xi,\eta),  Y(\xi,\eta))$ are uniformly $C^{(1-\nu)/2}$ continuous and that $W(\xi,\eta)$ is uniformly $C^{(2-\nu)/6}$ continuous on $MNO$ for any $\nu\in(0,1)$. Consequently, $( X(\xi,\eta),  Y(\xi,\eta), W(\xi,\eta))$ is uniformly $C^{\mu}$ continuous for all $\mu\in(0,1/3)$.

   Therefore, in view of (\ref{Inv1_5}) and (\ref{main_8}), the functions $\theta(\xi,\eta)$ and $\bar\omega(\xi,\eta)$ are uniformly $C^{1,\mu}$ continuous on the entire region $MNO$ for any $\mu\in(0,1/3)$. The uniform regularity of $c(\xi,\eta)$ follows directly from its representation given in (\ref{Pre1_8}).

   Finally, concerning the sonic curve $\widehat{MN}$, we observe by once again applying (\ref{Inv1_5}) that
   \begin{equation*}\label{main_13}
       \bar\omega_{\xi}^2+\bar\omega_{\eta}^2=\left(\sqrt{1-t^2}(1+\kappa(\tau)(1-t^2))W+\frac{1-t^2}{2} \right)^2+\frac{(1+\kappa(\tau)(1-t^2))^2}{4}( X- Y)^2,
   \end{equation*}
   which combined with Lemmas \ref{Sol_3_l1} and \ref{Sol_3_l4}, implies that
   \begin{equation*}\label{main_14}
       0<C_1\leq \bar\omega_{\xi}^2+\bar\omega_{\eta}^2\leq C_2,
   \end{equation*}
   for some uniform positive constants $C_{1}$ and $C_{2}$. Consequently, the curve $\widehat{MN}$ is $C^{1,\mu}$ continuous for any $\mu\in(0,1/3)$ and the proof of the lemma is complete.
\end{proof}
\begin{corollary}
    The function $\theta(\xi,\eta)$ has improved regularity. In particular, $\theta(\xi,\eta)$ is uniformly $C^{1,\bar\mu}$ continuous on the entire region $MNO$ for any $\bar\mu\in(0,1/2)$.
\end{corollary}
Along the curve $\widehat{NO}$, we have the following lemma
\begin{lemma}\label{Main_l2}
    The curve $\widehat{NO}$ specified in \eqref{main_4} constitutes a positive characteristic curve of the system (\ref{main_5}).
\end{lemma}
\begin{proof}
    Let $\hat z(\xi,\eta)=z^{+}(\hat t(\xi,\eta); \bar{\delta}, \tilde z(\bar{\delta}))$. Differentiating it with respect to its first argument and applying (\ref{main_5}), we obtain
    \begin{equation*}\label{main_15}
        \hat z_\xi+\hat z_\eta \frac{d\eta}{d\xi}=\frac{cf\hat{t}^2 X}{ X+\hat{t}g}\left( \hat t_\xi+\hat t_\eta \frac{d\eta}{d\xi}\right).
    \end{equation*}
    Utilizing (\ref{main_1}) we have
    \begin{equation*}\label{main_16}
        \left(\xi_t+\frac{cft^2 X}{ X+tg}\xi_z\right) \frac{d\eta}{d\xi}=\eta_t+\frac{cf\hat{t}^2 X}{ X+\hat{t}g}\eta_z.
    \end{equation*}
    Further (\ref{Inv1_4}), (\ref{main_7}) and (\ref{Sol1_10}) yields
    \begin{equation*}\label{main_17}
        \frac{d\eta}{d\xi}=\tan\alpha =\lambda_+,
    \end{equation*}
    which implies that the curve $\widehat{NO}$ is a positive characteristic curve of the system (\ref{Pre2_3}). This completes the proof of the lemma.
\end{proof}
\begin{lemma}\label{Main_l3}
    The functions $(c(\xi,\eta),\theta(\xi,\eta),\bar\omega(\xi,\eta))$ defined in (\ref{main_2}) satisfy the system (\ref{Pre2_3}).
\end{lemma}
\begin{proof}
    We verify only the first equation in system (\ref{Pre2_3}), since the second equation can be established in an analogous manner. By direct application of (\ref{Pre2_1}), (\ref{Inv1_5}), (\ref{main_7}) and (\ref{main_8}), it follows that
    \begin{equation}\label{main_18}
        \begin{split}
            \bar\partial^+ \theta =&\cos\alpha \theta_\xi+\sin\alpha\theta_\eta\\
            =&\cos\alpha\left( -\frac{\kappa(\tau)\cos\theta\sqrt{1-t^2}}{2}( X- Y)+\frac{\sin\theta}{\sqrt{1-t^2}}\left( \frac{1-t^2}{c}+\frac{1}{2}\kappa(\tau)\sqrt{1-t^2}\cos\omega( X+ Y)\right)\right)\\
            & +\sin\alpha \left( -\frac{\kappa(\tau)\sin\theta\sqrt{1-t^2}}{2}( X- Y)-\frac{\cos\theta}{\sqrt{1-t^2}}\left( \frac{1-t^2}{c}+\frac{1}{2}\kappa(\tau)\sqrt{1-t^2}\cos\omega( X+ Y)\right)\right)\\
            =& -\frac{\kappa(\tau)\sqrt{1-t^2}}{2}( X- Y)\left(\cos\theta\cos(\theta+\omega)+\sin\theta\sin(\theta+\omega) \right)+\frac{\sqrt{1-t^2}}{c}\left(\sin\theta\cos(\theta+\omega)-\cos\theta\sin(\theta+\omega) \right) \\
            &+ \frac{1}{2}\kappa(\tau)\cos\omega( X+ Y) \left(\sin\theta\cos(\theta+\omega)-\cos\theta \sin(\theta +\omega) \right)\\
            =&-\frac{\kappa(\tau)\sqrt{1-t^2}}{2}( X- Y)\cos\omega -\frac{1-t^2}{c}-\frac{1}{2}\kappa(\tau)\cos\omega( X+ Y)\sqrt{1-t^2}\\
            =&- \kappa(\tau)\bar\omega\cos\omega X-\frac{\bar\omega^2}{c}.
        \end{split}
    \end{equation}
    Similarly we evaluate
    \begin{equation}\label{main_19}
        \begin{split}
           \bar \partial^+\bar\omega =&\cos\alpha \bar\omega_\xi+\sin\alpha\bar\omega_\eta\\
           =& \cos\alpha\left(\cos\theta\sqrt{1-t^2} (1+\kappa(\tau)(1-t^2)){W}+\frac{\cos\theta(1-t^2)}{c}-\sin\theta\frac{(1+\kappa(\tau)(1-t^2))}{2}(  X- Y) \right)\\ 
           &+\sin\alpha \left(\sin\theta\sqrt{1-t^2}(1+\kappa(\tau)(1-t^2)){W}+\frac{\sin\theta(1-t^2)}{c}+\cos\theta\frac{(1+\kappa(\tau)(1-t^2))}{2}(  X- Y) \right)\\
           =& \sqrt{1-t^2} (1+\kappa(\tau)(1-t^2)){W}\cos(\alpha-\theta)+\frac{1-t^2}{c}\cos(\alpha-\theta)-\frac{(1+\kappa(\tau)(1-t^2))}{2}(  X- Y)\sin(\theta-\alpha)\\
           =&\sqrt{1-t^2} (1+\kappa(\tau)(1-t^2))\frac{ X+ Y}{2}\cos\omega+\frac{1-t^2}{c}\cos\omega+\frac{(1+\kappa(\tau)(1-t^2))}{2}(  X- Y)\sqrt{1-t^2}\\
           =& \cos\omega\bar\omega (1+\kappa(\tau)(1-t^2)) X+\frac{\bar\omega^2}{c}\cos\omega.
        \end{split}
    \end{equation}
    Applying (\ref{main_18}) and (\ref{main_19}), we compute
    \begin{equation*}\label{main_20}
        \begin{split}
            \bar\partial^+ \theta + \frac{\kappa(\tau)\cos\omega}{1+\kappa(\tau)\bar\omega^2}\bar \partial^+\bar\omega =& - \kappa(\tau)\bar\omega\cos\omega X-\frac{\bar\omega^2}{c} + \frac{\kappa(\tau)\cos\omega}{1+\kappa(\tau)\bar\omega^2}\left( \cos\omega\bar\omega (1+\kappa(\tau)(1-t^2)) X+\frac{\bar\omega^2}{c}\cos\omega\right)\\
            =& - \kappa(\tau)\bar\omega\cos\omega X-\frac{\bar\omega^2}{c} + \kappa(\tau)\cos^2\omega\bar\omega X+\frac{\bar\omega^2}{c}\cdot \frac{\kappa(\tau)\cos^2\omega}{1+\kappa(\tau)\bar\omega^2}\\
            =& \frac{\bar\omega^2}{c}\cdot \frac{\kappa(\tau)-1-2\kappa(\tau)\bar\omega^2}{1+\kappa(\tau)\bar\omega^2}.
        \end{split}
    \end{equation*}
    This yields the first equation in system (\ref{Pre2_3}), thereby completing the proof of the lemma.
\end{proof}
Combining Lemmas \ref{Main_l1}-\ref{Main_l3}, we complete proof of the Theorem \ref{Pre_th1}. Using (\ref{main_2}) and (\ref{Pre1_8}), we now define functions $(c,u,v)(\xi,\eta)$ as follows:
\begin{equation*}\label{main_21}
    \begin{split}
        u(\xi,\eta)=\xi+\frac{c(\xi,\eta)\cos\theta(\xi,\eta)}{\bar\omega(\xi,\eta)},\; v(\xi,\eta)=\eta+\frac{c(\xi,\eta)\sin\theta(\xi,\eta)}{\bar\omega(\xi,\eta)},\; c(\xi,\eta)=\frac{K(\gamma+1)\tau^2(\xi,\eta)}{(\tau(\xi,\eta)-b)^{\gamma+2}}-\frac{2a}{b}.
    \end{split}
\end{equation*}
It follows directly from Lemma \ref{Main_l1} that the functions $(\rho,u,v)(\xi,\eta)$ are uniformly $C^{1,\mu}$ continuous on the entire region $MNO$ for any $\mu\in(0,1/3)$. Moreover, we may verify that the functions $(\rho,u,v)(\xi,\eta)$ defined above satisfy the two-dimensional isentropic pseudo steady Euler equations (\ref{Pre1_1}). %This completes the proof of Theorem \ref{Pre_th1}.
%%%%%%%%%%%%%%%%%%%%%%%%%%%%%%%%%%%%%%%%%%%%%%%%%%%%%%%%%%%%%%%%%%%%%%%%%%%%%%%%%%%%%%%%%%%%%%%%%%%%%%%%%%%
\section{Conclusions}

In this work, we investigated the two-dimensional four-state Riemann problem for the isentropic compressible Euler equations with a polytropic van der Waals equation of state. We identified conditions under which a supersonic–sonic patch develops along a pseudo-streamline and analyzed the local structure of the flow near the sonic transition. By introducing self-similar variables, the governing system was shown to exhibit a mixed-type character with degeneracy arising on the sonic boundary. To overcome this difficulty we employed a characteristic decomposition combined with a partial hodograph transformation which enabled the construction of a smooth supersonic solution extending up to the sonic curve. Furthermore, we established uniform regularity estimates for both the solution and the resulting sonic boundary. 
Moreover, we prove that the sonic curve is uniformly $C^{1,\mu}$-continuous and the supersonic solution is uniformly $C^{1,\mu}$-continuous up to the sonic boundary for any $\mu\in(0,\frac{1}{3})$. These results demonstrate that the van der Waals equation of state can be incorporated naturally into the existing analytical framework for supersonic–sonic patch problems, thereby extending the theory from idealized polytropic gases to a more realistic class of non-ideal fluid models. The approach developed here can be adapted to study supersonic–sonic structures arising in general equation of state and related multidimensional flow configurations such as magnetohydrodynamics, relativistic Euler equations and etc.

\section*{Acknowledgements}
\textit{The first author acknowledges the financial support provided by the Indian Institute of Technology Kharagpur. The second author (TRS) expresses his gratitude towards ANRF, DST, Government of India (Ref. No. CRG/2022/006297) for its financial support through the core research grant.}


\begin{thebibliography}{plain}
    \bibitem{Jli1998}
{Li, Jiequan and Zhang, Tong and Yang, Shuli,
The two-dimensional Riemann problem in gas dynamics,
Pitman Monographs and Surveys in Pure and Applied Mathematics, Vol. 98,
Longman, Harlow, 1998}.

\bibitem{yuxizhang2001}
{Zheng, Yuxi,
Systems of conservation laws: Two-dimensional Riemann problems,
Progress in Nonlinear Differential Equations and their Applications, Vol. 38,
Birkh\"auser Boston, Inc., Boston, MA, 2001}.

\bibitem{Courant1948}
{Courant, R. and Friedrichs, K. O.,
Supersonic Flow and Shock Waves,
Interscience Publishers, New York, 1948}.

\bibitem{callen1998thermodynamics}
{Callen, Herbert B.,
Thermodynamics and an Introduction to Thermostatistics,
American Association of Physics Teachers, College Park, MD, 1998}.

\bibitem{liboundaryvalue}
{Li, Ta Tsien and Yu, Wen Ci,
Boundary value problems for quasilinear hyperbolic systems,
Duke University Mathematics Series, Vol. V,
Duke University Mathematics Department, Durham, NC, 1985}.

\bibitem{Pandey1}
{Pandey, Anamika and Barthwal, Rahul and Raja Sekhar, T.,
Construction of solutions of the Riemann problem for a two-dimensional Keyfitz-Kranzer type model governing a thin film flow,
Applied Mathematics and Computation, Vol. 498, 2025, Paper No. 129378}.

\bibitem{Pandey2}
{Pandey, Anamika and Raja Sekhar, T.,
Formation of vacuum state and delta-shock in the solution of two-dimensional Riemann problem for zero pressure gas dynamics, Accepted for publication in Studies in Applied Mathematics, 2025}.

\bibitem{Rahul1}
{Barthwal, Rahul and Raja Sekhar, T.,
Simple waves for two-dimensional magnetohydrodynamics with extended Chaplygin gas,
Indian Journal of Pure and Applied Mathematics, Vol. 53, No. 2, 2022, pp. 542--549}.

\bibitem{Rahul2}
{Barthwal, Rahul and Raja Sekhar, T.,
Existence of solutions to gas expansion problem through a sharp corner for two-dimensional Euler equations with general equation of state,
Studies in Applied Mathematics, Vol. 151, No. 1, 2023, pp. 141--170}.

\bibitem{Rahul3}
{Barthwal, Rahul and Raja Sekhar, T.,
Existence and regularity of solutions of a supersonic-sonic patch arising in axisymmetric relativistic transonic flow with general equation of state,
Journal of Mathematical Analysis and Applications, Vol. 523, No. 2, 2023, Paper No. 127022}.

\bibitem{Rahul4}
{Barthwal, Rahul and Raja Sekhar, T.,
On the existence and regularity of solutions of semihyperbolic patches to two-dimensional Euler equations with van der Waals gas,
Studies in Applied Mathematics, Vol. 148, No. 2, 2022, pp. 543--576}.

\bibitem{Mzafar2023}
{Srivastava, H. and Zafar, M.,
Degenerate Cauchy-Goursat problem for two-dimensional steady isentropic Euler system with van der Waals gas,
Studies in Applied Mathematics, Vol. 151, No. 4, 2023, pp. 1525--1549}.

\bibitem{Mzafar2025}
{Srivastava, H. and Zafar, M.,
Existence and regularity of sonic-supersonic solutions for two-dimensional steady magnetohydrodynamics system with van der Waals fluid,
Mathematical Methods in the Applied Sciences, 2025}.

\bibitem{cshen}
{Lai, Geng and Shen, Chun,
Characteristic decompositions and boundary value problems for two-dimensional steady relativistic Euler equations,
Mathematical Methods in the Applied Sciences, Vol. 37, No. 1, 2014, pp. 136--147}.

\bibitem{conjecture1990}
{Zhang, Tong and Zheng, Yuxi,
Conjecture on the structure of solutions of the Riemann problem for two-dimensional gas dynamics systems,
SIAM Journal on Mathematical Analysis, Vol. 21, No. 3, 1990, pp. 593--630}.

\bibitem{pdlax}
{Lax, Peter D. and Liu, Xu-Dong,
Solution of two-dimensional Riemann problems of gas dynamics by positive schemes,
SIAM Journal on Scientific Computing, Vol. 19, No. 2, 1998, pp. 319--340}.

\bibitem{Alexander2017}
{Kurganov, Alexander and Prugger, Martina and Wu, Tong,
Second-order fully discrete central-upwind scheme for two-dimensional hyperbolic systems of conservation laws,
SIAM Journal on Scientific Computing, Vol. 39, No. 3, 2017, pp. A947--A965}.

\bibitem{Jli2009}
{Li, Jiequan and Sheng, Wancheng and Zhang, Tong and Zheng, Yuxi,
Two-dimensional Riemann problems: from scalar conservation laws to compressible Euler equations,
Acta Mathematica Scientia Series B, Vol. 29, No. 4, 2009, pp. 777--802}.

\bibitem{Jli2010}
{Li, Jiequan and Zheng, Yuxi,
Interaction of four rarefaction waves in the bi-symmetric class of the two-dimensional Euler equations,
Communications in Mathematical Physics, Vol. 296, No. 2, 2010, pp. 303--321}.

\bibitem{Schulz1993}
{Schulz-Rinne, Carsten W. and Collins, James P. and Glaz, Harland M.,
Numerical solution of the Riemann problem for two-dimensional gas dynamics,
SIAM Journal on Scientific Computing, Vol. 14, No. 6, 1993, pp. 1394--1414}.

\bibitem{Glimm2008}
{Glimm, James and Ji, Xiaomei and Li, Jiequan and Li, Xiaolin and Zhang, Peng and Zhang, Tong and Zheng, Yuxi,
Transonic shock formation in a rarefaction Riemann problem for the two-dimensional compressible Euler equations,
SIAM Journal on Applied Mathematics, Vol. 69, No. 3, 2008, pp. 720--742}.

\bibitem{simplewave2006}
{Li, Jiequan and Zhang, Tong and Zheng, Yuxi,
Simple waves and a characteristic decomposition of the two-dimensional compressible Euler equations,
Communications in Mathematical Physics, Vol. 267, No. 1, 2006, pp. 1--12}.

\bibitem{gasexpansion2001}
{Li, Jiequan,
On the two-dimensional gas expansion for compressible Euler equations,
SIAM Journal on Applied Mathematics, Vol. 62, No. 3, 2001, pp. 831--852}.

\bibitem{interaction2009}
{Li, Jiequan and Zheng, Yuxi,
Interaction of rarefaction waves of the two-dimensional self-similar Euler equations,
Archive for Rational Mechanics and Analysis, Vol. 193, No. 3, 2009, pp. 623--657}.

\bibitem{chdecomposition}
{Li, Jiequan and Yang, Zhicheng and Zheng, Yuxi,
Characteristic decompositions and interactions of rarefaction waves of two-dimensional Euler equations,
Journal of Differential Equations, Vol. 250, No. 2, 2011, pp. 782--798}.

\bibitem{LaiGeng2019}
{Lai, Geng,
Global solutions to a class of two-dimensional Riemann problems for the Euler equations with a general equation of state,
Indiana University Mathematics Journal, Vol. 68, No. 5, 2019, pp. 1409--1464}.

\bibitem{glai2021}
{Lai, Geng and Sheng, Wancheng,
Two-dimensional pseudosteady flows around a sharp corner,
Archive for Rational Mechanics and Analysis, Vol. 241, No. 2, 2021, pp. 805--884}.

\bibitem{wsheng2018}
{Sheng, Wancheng and You, Shouke,
Interaction of a centered simple wave and a planar rarefaction wave of the two-dimensional Euler equations for pseudo-steady compressible flow,
Journal de Math\'ematiques Pures et Appliqu\'ees, Vol. 114, 2018, pp. 29--50}.

\bibitem{songzheng2009}
{Song, Kyungwoo and Zheng, Yuxi,
Semi-hyperbolic patches of solutions of the pressure gradient system,
Discrete and Continuous Dynamical Systems Series A, Vol. 24, No. 4, 2009, pp. 1365--1380}.

\bibitem{Tesdall2006}
{Tesdall, Allen M. and Sanders, Richard and Keyfitz, Barbara L.,
The triple point paradox for the nonlinear wave system,
SIAM Journal on Applied Mathematics, Vol. 67, No. 2, 2006, pp. 321--336}.

\bibitem{Tesdall2008}
{Tesdall, Allen M. and Sanders, Richard and Keyfitz, Barbara L.,
Self-similar solutions for the triple point paradox in gasdynamics,
SIAM Journal on Applied Mathematics, Vol. 68, No. 5, 2008, pp. 1360--1377}.

\bibitem{LMingjie2011}
{Li, Mingjie and Zheng, Yuxi,
Semi-hyperbolic patches of solutions to the two-dimensional Euler equations,
Archive for Rational Mechanics and Analysis, Vol. 201, No. 3, 2011, pp. 1069--1096}.

\bibitem{yanbohu2014}
{Hu, Yanbo and Wang, Guodong,
Semi-hyperbolic patches of solutions to the two-dimensional nonlinear wave system for Chaplygin gases,
Journal of Differential Equations, Vol. 257, No. 5, 2014, pp. 1567--1590}.

\bibitem{yanbohu2018}
{Hu, Yanbo and Li, Tong,
An improved regularity result of semi-hyperbolic patch problems for the two-dimensional isentropic Euler equations,
Journal of Mathematical Analysis and Applications, Vol. 467, No. 2, 2018, pp. 1174--1193}.

\bibitem{chenj2020}
{Chen, Jianjun and Lai, Geng,
The regularity of semi-hyperbolic patches near sonic curves for the two-dimensional compressible magnetohydrodynamic equations,
ZAMM Zeitschrift f\"ur Angewandte Mathematik und Mechanik, Vol. 100, No. 11, 2020}.

\bibitem{Yongqiang2023}
{Fan, Yongqiang and Guo, Lihui and Hu, Yanbo and You, Shouke,
Semi-hyperbolic patch characterized by two-dimensional steady relativistic Euler equations,
Journal of Differential Equations, Vol. 354, 2023, pp. 264--295}.

\bibitem{CJianjun2025}
{Chen, Jianjun and Zhang, Yuqi and Li, Shuangrong,
The regularity of semi-hyperbolic patches of solutions to the two-dimensional compressible Euler equations in magnetohydrodynamics,
Journal of Mathematical Analysis and Applications, Vol. 545, No. 2, 2025, Paper No. 129242}.

\bibitem{hu2024supersonic}
{Hu, Yanbo and Wang, Guodong,
On a supersonic-sonic patch arising from the two-dimensional Riemann problem of the compressible Euler equations,
Proceedings of the Royal Society of Edinburgh Section A Mathematics, 2024}.

\bibitem{MR3342411}
{Lai, Geng,
On the expansion of a wedge of van der Waals gas into a vacuum,
Journal of Differential Equations, Vol. 259, No. 3, 2015, pp. 1181--1202}.

\bibitem{lgeng2015}
{Lai, Geng and Sheng, Wancheng,
Centered wave bubbles with sonic boundary of pseudosteady Guderley Mach reflection configurations in gas dynamics,
Journal de Math\'ematiques Pures et Appliqu\'ees (9), Vol. 104, No. 1, 2015, pp. 179--206}.
\end{thebibliography}
\end{document}